\documentclass[10pt]{article}

\usepackage{amssymb}
\usepackage{amsmath}
\usepackage{theorem}
\usepackage{epsfig}
\usepackage{verbatim}
\usepackage{graphicx}
\usepackage{subfigure}
\usepackage{cancel}
\usepackage{appendix}

\usepackage{color}

\newtheorem{theorem}{Theorem}[section]
\newtheorem{lemma}[theorem]{Lemma}

\newtheorem{corollary}[theorem]{Corollary}

\newtheorem{defi}[theorem]{Definition}
\newtheorem{rem}[theorem]{Remark}

\newcommand{\R}{\mathbb{R}}

\newcommand{\Z}{\mathbb{Z}}

\newcommand{\al}{\alpha}

\newcommand{\ep}{\varepsilon}

\newcommand{\pa}{\partial}

\newenvironment{proof}{\begin{trivlist} \item[] {\em Proof:}}{\hfill $\Box$
                       \end{trivlist}}

\newenvironment{proofthm}[1]{\begin{trivlist} \item[] {\em Proof of Theorem \ref{#1}:}}{\hfill $\Box$
                       \end{trivlist}}

\makeatletter
\renewcommand*\l@section{\@dottedtocline{1}{0em}{1.5em}}
\renewcommand*\l@subsection{\@dottedtocline{2}{1.5em}{2.3em}}
\renewcommand*\l@subsubsection{\@dottedtocline{3}{3.8em}{3.7em}}
\makeatother

\numberwithin{equation}{section}

\allowdisplaybreaks[4]

\begin{document}

\title{Uniformly rotating analytic global patch solutions for active scalars}

\author{Angel Castro, Diego C\'ordoba and Javier G\'omez-Serrano}

\maketitle

\begin{abstract}

We show that there exists a family of analytic convex global rotating solutions for the vortex patch equations, bifurcating from ellipses. As a byproduct, the analyticity proof can also be adapted to the rotating patch solutions bifurcating from disks (also known as V-states) for both the Euler and the generalized surface quasi-geostrophic equation. \\

\vskip 0.3cm
\textit{Keywords: bifurcation theory, analyticity, Crandall-Rabinowitz, V-states, patches, surface quasi-geostrophic, Euler}

\end{abstract}


\section{Introduction}

Our goal in this article is to show the existence of a new family of uniformly rotating global solutions of the vortex patch equation. We start with the 2D Euler equations in vorticity form:

\begin{align*}
\omega_t + u \cdot \nabla \omega  & = 0 , \quad (x,t) \in \mathbb{R}^{2} \times \mathbb{R} \\
u(x,t) & = -\nabla^{\perp}(-\Delta)^{-1} \omega,
\end{align*}

and we will consider weak solutions in the form of a patch, that is, solutions for which $\omega$ is
 a step function:

  \begin{align}
  \omega(x,t) =
  \left\{
 \begin{array}{ll}
   \omega_1, \text{ if } \ \ x \in \Omega(t) \\ 
   \omega_2, \text{ if }  \ \ x \in \Omega(t) ^c, \\
   \end{array}
  \right.
\end{align}

 where $\Omega(0)$ is given by the initial distribution of $\omega$, $\omega_1$ and $\omega_2$ are constants, and $\Omega(t)$ is the evolution of $\Omega(0)$ under the velocity field $u$. The problem can be reduced to an evolution equation for $\partial \Omega(t)$.

Yudovich proved the global existence and uniqueness of weak solutions in $L^{1} \cap L^{\infty}$ of the 2D Euler in vorticity formulation \cite{Yudovich:Nonstationary-ideal-incompressible}. Chemin, in \cite{Chemin:persistance-structures-fluides-incompressibles} showed by means of paradifferential calculus the preservation of $\mathcal{C}^{1, \gamma}$ regularity of the boundary of the patch $\pa \Omega(t)$. Another proof of that result, highlighting the extra cancellation on semi spheres of even kernels, can be found in \cite{Bertozzi-Constantin:global-regularity-vortex-patches} by Bertozzi and Constantin. Serfati, in \cite{Serfati:preuve-directe-existence-globale-patches} provided another one, giving a fuller characterization of the velocity gradient's regularity.

  In recent years, Denisov has studied the process of merging for the vortex patch problem. In this case, the collapse in a point can not happen in finite time since the distance between the two patches can decay at most as fast as a double exponential. Denisov proves in \cite{Denisov:sharp-corner-euler-patches} that this bound is sharp if one is allowed to modify slightly the velocity by superimposing a smooth background incompressible flow. See also \cite{Denisov:centrally-symmetric-vstates-model}.

However, there is a family of global solutions that move with constant, both in time and space, angular velocity, called V-states. Deem and Zabusky in \cite{Deem-Zabusky:vortex-waves-stationary} were the first to compute them numerically, and later other authors have improved the methods and numerically computed a bigger class (see \cite{Wu-Overman-Zabusky:steady-state-Euler-2d,Elcrat-Fornberg-Miller:stability-vortices-cylinder,LuzzattoFegiz-Williamson:efficient-numerical-method-steady-uniform-vortices,Saffman-Szeto:equilibrium-shapes-equal-uniform-vortices} for a small sample of them).

Concerning proofs of the existence and regularity of V-states, Burbea \cite{Burbea:motions-vortex-patches} used a conformal mapping and bifurcation theory to show their existence. Hmidi, Mateu and Verdera in \cite{Hmidi-Mateu-Verdera:rotating-vortex-patch} showed that the family of V-states has $C^{\infty}$ boundary regularity and is convex. In another paper \cite{Hmidi-delaHoz-Mateu-Verdera:doubly-connected-vstates-euler}, they studied the V-state existence for the case of doubly connected domains. See also \cite{Hmidi-Mateu-Verdera:rotating-doubly-connected-vortices}.

It is known since Kirchhoff \cite{Kirchhoff:vorlesungen-math-physik} that ellipses are a family of rotating solutions for the vortex patch equations. More precisely, an ellipse of semiaxes $a$ and $b$, rotates with uniform velocity $\Omega = \frac{ab}{(a+b)^{2}}$. Love \cite{Love:stability-ellipses} established linear stability for ellipses of aspect ratio bigger than $\frac{1}{3}$ and linear instability for ellipses of aspect ratio smaller than $\frac{1}{3}$. Most of the efforts have been devoted to establish nonlinear stability and instability in the range predicted by the linear part. Wan \cite{Wan:stability-rotating-vortex-patches}, and Tang \cite{Tang:nonlinear-stability-vortex-patches} proved the nonlinear stable case, whereas Guo et al. \cite{Guo-Hallstrom-Spirn:dynamics-unstable-kirchhoff-ellipse} settled the nonlinear unstable one. See also \cite{Constantin-Titi:evolution-nearly-circular-vortex-patches}.

Our first theorem shows the existence of a family of analytic convex global rotating solutions that bifurcate from ellipses of certain aspect ratios. Kamm, in his PhD thesis \cite{Kamm:thesis-shape-stability-patches}, computed numerically the first 3 bifurcations and the first 3 values of the aspect ratio at which the bifurcations occur. His results agree perfectly with our analytic formula. We remark that all of the aspect ratios (other than the first one - which is $\frac{1}{3}$ -) from which we bifurcate are smaller than $\frac{1}{3}$ and lie in the unstable zone (see Lemma \ref{lemaraices}). A small sample of other numerical studies of ellipses, either perturbed or unperturbed, or concerning their stability are  \cite{Dritschel:stability-energetics-corotating-vortices,Dritchel:nonlinear-evolution-rotating-patches,Mitchell-Rossi:evolution-kirchhoff-elliptic-vortices,Cerretelli-Williamson:new-family-vortices}.


The existence result, in H\"older spaces, was announced by Hmidi in a talk given at ICMAT in May 2014 and recently appeared in \cite{Hmidi-Mateu:bifurcation-kirchhoff-ellipses}.

 The evolution equation for the interface of a vortex patch, which we parametrize as a $2 \pi$ periodic curve $z(x)$, can be written as
 \begin{align}
\label{Ecuacion-vortex-patch}
 \partial_t z(x,t) = \frac{\omega_2 - \omega_1}{4\pi} \int_{0} ^{ 2 \pi} (\partial_x z (x,t) - \partial_x z(x-y,t) )\log(|z(x,t)-z(x-y,t)|^{2}) dy,
 \end{align}

 since we can add terms in the tangential direction without changing the evolution of the patch.

We will also consider the generalized surface-quasigeostrophic equation (gSQG):
\begin{align*}
\left\{ \begin{array}{ll}
\partial_{t}\theta+u\cdot\nabla\theta=0,\quad(t,x)\in\mathbb{R}_+\times\mathbb{R}^2, &\\
u=-\nabla^\perp(-\Delta)^{-1+\frac{\alpha}{2}}\theta,\\
\theta_{|t=0}=\theta_0,
\end{array} \right.
\end{align*}

where $\alpha \in (0,2)$. The case $\alpha = 1$ corresponds to the surface quasi-geostrophic (SQG) equation and the limiting case $\alpha = 0$ refers to the 2D incompressible Euler equation discussed before. $\alpha = 2$ produces stationary solutions. As before, we will work in the patch setting.

%
%
%
%
%

  
  In this setting, local existence of $C^{\infty}$ solutions was proved by Rodrigo in \cite{Rodrigo:evolution-sharp-fronts-qg} for $0 < \al \leq 1$ and for Sobolev regularity by Gancedo in \cite{Gancedo:existence-alpha-patch-sobolev} for $0 < \al \leq 1$ and for $1 < \al < 2$ by Chae et al. in \cite{Chae-Constantin-Cordoba-Gancedo-Wu:gsqg-singular-velocities}.

Looking for candidates for singularities, there have been several numerical experiments in that direction. The evolution of two colliding patches was studied in \cite{Cordoba-Fontelos-Mancho-Rodrigo:evidence-singularities-contour-dynamics} by C\'ordoba et al. for a broad spectrum of $\alpha$. They exhibit numerical evidence suggesting an asymptotically self-similar singular scenario in which the distance between both patches goes to zero in finite time while simultaneously the curvature of the boundaries blows up. This does not contradict the theorem of Gancedo and Strain \cite{Gancedo-Strain:absence-splash-muskat-SQG} who proved that no splash singularity can be formed, i.e., two interfaces can not collapse in a point, if the interfaces remain smooth.  Another possible scenario for singularities is the one presented by Scott and Dritschel \cite{Scott-Dritschel:self-similar-sqg}, where taking as initial condition an elliptical patch with a big aspect ratio between its axes may develop a self-similar singularity with curvature blowup for the case $\al = 1$. Castro et al. (\cite{Castro-Cordoba-GomezSerrano-MartinZamora:remarks-geometric-properties-sqg}) proved that, unlike in the vortex patch, elliptical patches are not rotating solutions for $\al > 0$, as well as the existence of convex solutions that lose their convexity in finite time. Scott \cite{Scott:scenario-singularity-quasigeostrophic} pointed out that small perturbations of thin strips may lead to a self similar cascade of instabilities, leading to a possible arc chord blow up.

In \cite{Hassainia-Hmidi:v-states-generalized-sqg} and \cite{delaHoz-Hassainia-Hmidi:doubly-connected-vstates-gsqg}, the authors showed the existence of V-states for $0 < \al < 1$, in the simply connected and doubly connected case respectively, leaving open the regularity of the solutions. This question was addressed in the simply connected case in \cite{Castro-Cordoba-GomezSerrano:existence-regularity-gsqg} for the full range $0 < \al < 2$, proving existence of solutions for $1 \leq \al < 2$ and $C^{\infty}$ regularity of the boundary for $0 < \al < 2$.

The second theorem of this paper shows that using the same techniques as for the vortex patch equation, one can adapt the spaces to get existence of analytic V-states, improving thus the $C^{\infty}$ regularity result from \cite{Hmidi-Mateu-Verdera:rotating-vortex-patch,Castro-Cordoba-GomezSerrano:existence-regularity-gsqg}.

 The evolution equation for the interface of an $\alpha-$ patch, which we parametrize as a $2 \pi$ periodic curve $z(x)$, can be written as
 \begin{align}
\label{Ecuacion-alpha-patch}
 \partial_t z(x,t) = -(\theta_2 - \theta_1)C(\al) \int_{0} ^{ 2 \pi} \frac{ \partial_x z (x,t) - \partial_x z(x-y,t) }{ \vert z(x,t) - z(x-y, t) \vert^{\alpha}} dy,
 \end{align}

where the normalizing constant $C(\al)$ is given by:

\begin{align*}
 C(\al) = \frac{1}{2\pi} \frac{\Gamma\left(\frac{\al}{2}\right)}{2^{1-\al}\Gamma\left(\frac{2-\al}{2}\right)}.
\end{align*}

\subsection{The contour equations}

Let $z(x,t) = (z_1(x,t),z_2(x,t))$ be the interface of the patch. We will derive the contour equation in two different settings: one bifurcating from ellipses and the vortex patch equation, the other one bifurcating from disks for the Euler and generalized SQG equation. 

Let us assume that $z(x,t)$ rotates with frequency $\Omega$ counterclockwise. Thus

\begin{align*}
z_t(x,t) = \Omega z^{\perp}(x,t),
\end{align*}

where for every $v = (v_1, v_2)$, $v^{\perp}$ is defined as $(-v_2,v_1)$. The equations such patch satisfies are

\begin{align*}
z_t(x,t) \cdot n & = u(z(x,t),t) \cdot n \\
\Omega \langle z^{\perp}(x,t), z_{x}^{\perp}(x,t) \rangle  = \Omega \langle z(x,t), z_{x}(x,t) \rangle & = \langle u(z(x,t),t), z_{x}^{\perp}(x,t) \rangle.
\end{align*}

Here, $n$ is the unitary normal vector and the tangential component of the velocity does not change the shape of the curve. Let us now deal with the elliptic case and parametrize $z(x,t)$ as

\begin{align*}
\left(
\begin{array}{c}
z_1(x,t) \\
z_2(x,t)
\end{array}
\right)
= 
\left(
\begin{array}{cc}
\cos(\Omega t) & \sin(\Omega t) \\
-\sin(\Omega t) & \cos(\Omega t)
\end{array}
\right)
\left(
\begin{array}{c}
(1 + R(x))\cos(x) \\
(r + R(x))\sin(x)
\end{array}
\right)
\end{align*}


where $r < 1$. We remark that the case $R(x) \equiv 0$ corresponds to an ellipse of semiaxes $1$ and $r$. We will set $\Omega = \frac{r}{(1+r)^{2}}$, so that $r$ becomes our bifurcation parameter. Thus, the question of finding a rotating global solution is reduced to find a zero of $F(r,R)$, where

\begin{align}
\label{funcionalelipses14}
F(r,R) & = \frac{r}{(1+r)^2}\left(\frac{r^2-1}{2}\sin(2x) + R'(x)(\cos^2(x)  + r \sin^{2}(x)) + (r-1)\sin(2x)R(x) + R(x)R'(x)\right) \nonumber \\
& + \frac{1}{4\pi}\sum_{i=1}^{2} F_{i}(r,R),
\end{align}

and the $F_{i}$ are


\begin{align*}
F_1(r,R) & = \int \log((\cos(x)-\cos(y) + R(x)\cos(x) - R(y)\cos(y))^{2} + (r(\sin(x)-\sin(y)) + R(x)\sin(x) - R(y)\sin(y))^{2})
 \\
& \times (\cos(x)(r+R(x))+\sin(x)R'(x))(\sin(y) + R(y)\sin(y) - \cos(y)R'(y)) dy\\
F_2(r,R) & = -\int \log((\cos(x)-\cos(y) + R(x)\cos(x) - R(y)\cos(y))^{2} + (r(\sin(x)-\sin(y)) + R(x)\sin(x) - R(y)\sin(y))^{2})
 \\
& \times (\sin(x)+R(x)\sin(x)-\cos(x)R'(x))(\cos(y)(r+R(y))+\sin(y)R'(y)) dy\\
\end{align*}

where the above integrals are performed on the torus. We remark that $F(r,0) \equiv 0$ for every $r$ (see \cite{Kirchhoff:vorlesungen-math-physik} and Section \ref{sectionpaso2elipses}). 

In the gSQG case, we will perturb the  disk by looking for solutions of the type:

\begin{align*}
\left(
\begin{array}{c}
z_1(x,t) \\
z_2(x,t)
\end{array}
\right)
= 
\left(
\begin{array}{cc}
\cos(\Omega t) & \sin(\Omega t) \\
-\sin(\Omega t) & \cos(\Omega t)
\end{array}
\right)
\left(
\begin{array}{c}
(1 + R(x))\cos(x) \\
(1 + R(x))\sin(x)
\end{array}
\right)
\end{align*}

Then, the question of finding a rotating global solution patch is reduced to find a zero of $F(\Omega,R)$, where

\begin{align}
\label{funcionalvstatesdiscos}
F(\Omega,R) = \Omega R'(x) - \sum_{i=1}^{3} F_{i}(R),
\end{align}

and the $F_{i}$ are

\begin{align}
F_1(R)=&\frac{1}{R(x)}C(\al)\int\frac{\sin(x-y)}{\left(\left(R(x)-R(y)\right)^2+4R(x)R(y)\sin^2\left(\frac{x-y}{2}\right)\right)^\frac{\al}{2}}
\left(R(x)R(y)+R'(x)R'(y)\right)dy,\\
F_2(R)=&C(\al)\int\frac{\cos(x-y)}{\left(\left(R(x)-R(y)\right)^2+4R(x)R(y)\sin^2\left(\frac{x-y}{2}\right)\right)^\frac{\al}{2}}
\left(R'(y)-R'(x)\right)dy,\\
F_3(R)=&\frac{R'(x)}{R(x)}C(\al)\int\frac{\cos(x-y)}{\left(\left(R(x)-R(y)\right)^2+4R(x)R(y)\sin^2\left(\frac{x-y}{2}\right)\right)^\frac{\al}{2}}
\left(R(x)-R(y)\right)dy,
\end{align}

For simplicity, from now on we will omit writing the domain of integration, which is always the torus.


\subsection{Functional spaces}

We refer to the space of analytic functions in the strip $|\Im(z)| \leq c$ as $\mathcal{C}_{w}(c)$. In our proofs, we will use the following analytic spaces. For $k \in \mathbb{Z}$:

\begin{align*}
X^{k}_{c} & = \left\{f(x) \in \mathcal{C}_{w}(c), f(x) = \sum_{j=1}^{\infty}a_{j} \cos(jx), \sum_{\pm}\int |f(x \pm ic)|^{2}dx + \sum_{\pm}\int |\pa^{k} f(x\pm ic)|^{2}dx < \infty\right\} \\
X^{k,odd}_{c} & = \left\{f(x) \in \mathcal{C}_{w}(c), f(x) = \sum_{j \geq 1, j \text{ odd }}^{\infty}a_{j} \cos(jx), \sum_{\pm}\int |f(x \pm ic)|^{2}dx + \sum_{\pm}\int |\pa^{k} f(x\pm ic)|^{2}dx < \infty\right\} \\
X^{k,even}_{c} & = \left\{f(x) \in \mathcal{C}_{w}(c), f(x) = \sum_{j \geq 2, j \text{ even }}^{\infty}a_{j} \cos(jx), \sum_{\pm}\int |f(x \pm ic)|^{2}dx + \sum_{\pm}\int |\pa^{k} f(x\pm ic)|^{2}dx < \infty\right\} \\
Y^{k}_{c} & = \left\{f(x) \in \mathcal{C}_{w}(c), f(x) = \sum_{j=1}^{\infty}a_{j} \sin(jx), \sum_{\pm}\int |f(x \pm ic)|^{2}dx + \sum_{\pm}\int |\pa^{k} f(x\pm ic)|^{2}dx < \infty\right\} \\
Y^{k,odd}_{c} & = \left\{f(x) \in \mathcal{C}_{w}(c), f(x) = \sum_{j \geq 1, j \text{ odd }}^{\infty}a_{j} \sin(jx), \sum_{\pm}\int |f(x \pm ic)|^{2}dx + \sum_{\pm}\int |\pa^{k} f(x\pm ic)|^{2}dx < \infty\right\} \\
Y^{k,even}_{c} & = \left\{f(x) \in \mathcal{C}_{w}(c), f(x) = \sum_{j \geq 2, j \text{ even }}^{\infty}a_{j} \sin(jx), \sum_{\pm}\int |f(x \pm ic)|^{2}dx + \sum_{\pm}\int |\pa^{k} f(x\pm ic)|^{2}dx < \infty\right\} \\
X^{k+\al}_{c} & = \left\{f(x) \in \mathcal{C}_{w}(c), f(x) = \sum_{j=1}^{\infty}a_{j} \cos(jx), \sum_{\pm}\int |f(x \pm ic)|^{2}dx + \sum_{\pm}\int |\pa^{k} f(x\pm ic)|^{2}dx \right. \\
& + \left.\sum_{\pm} \left\|\int_{\mathbb{T}} \frac{\pa^{k}f(x\pm ic -y)-\pa^{k}f(x \pm ic)}{\left|\sin\left(\frac{y}{2}\right)\right|^{1+\al}}dy\right\|_{L^2(x)} < \infty\right\}, \quad \al \in (0,1) \\
X^{k+\log}_{c} & = \left\{f(x) \in \mathcal{C}_{w}(c), f \in X^{k}_{c}, f(x) = \sum_{j=1}^{\infty}a_{j} \cos(jx), \sum_{\pm} \left\|\int_{\mathbb{T}} \frac{\pa^{k}f(x\pm ic -y)-\pa^{k}f(x \pm ic)}{\left(|\sin\left(\frac{y}{2}\right)\right|}dy\right\|_{L^2(x)} < \infty\right\} \\
X^{k+\log,m}_{c} & = \left\{f(x) \in \mathcal{C}_{w}(c), f \in X^{k}_{c}, f(x) = \sum_{j=1}^{\infty}a_{jm} \cos(jmx), \sum_{\pm} \left\|\int_{\mathbb{T}} \frac{\pa^{k}f(x \pm ic -y)-\pa^{k}f(x \pm ic)}{\left|\sin\left(\frac{y}{2}\right)\right|}dy\right\|_{L^2(x)} < \infty\right\}.
\end{align*}

The norm is given in the last two cases by the sum of the $X^{k}_{c}$-norm and the additional finite integral in the definition. 


\subsection{Theorems and outline of the proofs}

 The paper is organized as follows:

In Section \ref{sectionelipses}, we prove the following theorem:

\begin{theorem}
\label{teoremaexistenciaelipses}
 Let $k\geq 3, m \in \mathbb{N}, m > 2$ and let $0 < r(m) < 1$ be the (unique) positive solution of

 \begin{align*} -1 - 2 r + 2 m r - r^2 - \frac{(1-r)^{m}}{(1+r)^{m-2}} = 0.\end{align*}

Then, there exists a family of solutions $(r,R)$ and a $c > 0$, where $R(x) \in X^{k,even}_{c}$ or $R(x) \in X^{k,odd}_{c}$ if $m$ is even or odd resp., of the equation \eqref{funcionalelipses14} that bifurcate from the ellipse of semiaxes $1$ and $r(m)$.
\end{theorem}

Section \ref{sectiondiscos} is devoted to prove

\begin{theorem}
\label{teoremaexistenciadiscos}
 Let $k\geq 3, m \in \mathbb{N}, m \geq 2, 0 \leq \alpha < 2$ and let

 \begin{align*} \Omega_m = 
\left\{
\begin{array}{cc}
\displaystyle \frac{m-1}{2m} & \text{ if } \alpha = 0 \\
\displaystyle \frac{2}{\pi}\sum_{k=2}^{m}\frac{1}{2k-1} & \text{ if } \alpha = 1 \\
\displaystyle -2^{\al-1} \frac{\Gamma\left(1-\al\right)}{\left(\Gamma\left(1-\frac{\al}{2}\right)\right)^{2}}\left(\frac{\Gamma\left(1+\frac{\al}{2}\right)}{\Gamma\left(2-\frac{\al}{2}\right)} - \frac{\Gamma\left(m+\frac{\al}{2}\right)}{\Gamma\left(1+m-\frac{\al}{2}\right)}\right) & \text{ if } \alpha \neq \{0, 1\}
\end{array}
\right.
\end{align*} 

Then, there exists a family of $m$-fold solutions $(\Omega,R)$ and a $c > 0$, where $R(x)-1 \in X^{k+1,m}_{c}$ (for $\al < 1$), $R(x)-1 \in X^{k+1+\log,m}_{c}$ (for $\al = 1$) or $R(x)-1 \in X^{k+\al,m}_{c}$ (for $\al > 1$) of the equation \eqref{funcionalvstatesdiscos} with $0 \leq  \al < 2$ that bifurcate from the disk at $\Omega = \Omega_m$.
\end{theorem}

Both proofs are carried out by means of a combination of a Crandall-Rabinowitz's theorem and a priori estimates.

\section{Existence for the perturbation from ellipses}
\label{sectionelipses}

\begin{proofthm}{teoremaexistenciaelipses}
 The proof will be divided into 6 steps. These steps correspond to check the hypotheses of the Crandall-Rabinowitz theorem \cite{Crandall-Rabinowitz:bifurcation-simple-eigenvalues} for

\begin{align}
\label{funcionalelipses}
F(r,R) & = \frac{r}{(1+r)^2}\left(\frac{r^2-1}{2}\sin(2x) + R'(x)(\cos^2(x)  + r \sin^{2}(x)) + (r-1)\sin(2x)R(x) + R(x)R'(x)\right) + \frac{1}{4\pi}\sum_{i=1}^{2} F_{i}(r,R) \nonumber \\
& = F_{0}(r,R) + \frac{1}{4\pi}F_{1}(r,R) + \frac{1}{4\pi}F_{2}(r,R)
\end{align}

where

\scriptsize

\begin{align*}
F_1(r,R) & = \int \log\left(\frac{(\cos(x)-\cos(x-y) + R(x)\cos(x) - R(x-y)\cos(x-y))^{2} + (r(\sin(x)-\sin(x-y)) + R(x)\sin(x) - R(x-y)\sin(x-y))^{2})}{4\left|\sin\left(\frac{y}{2}\right)\right|^{2}}\right)
 \\
& \times (\cos(x)(r+R(x))+\sin(x)R'(x))(\sin(x-y) + R(x-y)\sin(x-y) - \cos(x-y)R'(x-y)) dy\\
& +\int \log\left(4\left|\sin\left(\frac{y}{2}\right)\right|^{2}\right)(\cos(x)(r+R(x))+\sin(x)R'(x))(\sin(x-y) + R(x-y)\sin(x-y) - \cos(x-y)R'(x-y)) dy\\
F_2(r,R) & = -\int \log\left(\frac{(\cos(x)-\cos(x-y) + R(x)\cos(x) - R(x-y)\cos(x-y))^{2} + (r(\sin(x)-\sin(x-y)) + R(x)\sin(x) - R(x-y)\sin(x-y))^{2}}{4\left|\sin\left(\frac{y}{2}\right)\right|^{2}}\right)
 \\
& \times (\sin(x)+R(x)\sin(x)-\cos(x)R'(x))(\cos(x-y)(r+R(x-y))+\sin(x-y)R'(x-y)) dy\\
& -\int \log\left(4\left|\sin\left(\frac{y}{2}\right)\right|^{2}\right) (\sin(x)+R(x)\sin(x)-\cos(x)R'(x))(\cos(x-y)(r+R(x-y))+\sin(x-y)R'(x-y)) dy\\
\end{align*}
\normalsize

The hypotheses are the following:

\begin{enumerate}

\item The functional $F$ satisfies $$F(r,R)\,:\, \R\times \{V^{\ep}\}\mapsto Y^{k-1}_{c},$$ where $V^{\ep}$ is the open neighbourhood of 0
$$V^{\ep}=\{ f\in X^{k}_{c}\,:\, ||f||_{X^{k}_{c}}<\ep\},$$
for all $0<\ep<\ep_{0}(m)$ and $k\geq 3$.

\item $F(r,0) = 0$ for every $r$.
\item The partial derivatives $F_{r}$, $F_{R}$ and $F_{r R}$ exist and are continuous.
\item Ker($\mathcal{F}$) and $Y^{k-1}_{c}$/Range($\mathcal{F}$) are one-dimensional, where $\mathcal{F}$ is the linearized operator around $R = 0$ at $r = r(m)$.
\item $F_{r R}(r(m),0)(h_0) \not \in$ Range($\mathcal{F}$), where Ker$(\mathcal{F}) = \langle h_0 \rangle$.
\item Step 1 can be applied to the spaces $X^{k,odd}_{c}$ ($X^{k,even}_{c}$) and $Y^{k-1,odd}_{c}$ ($X^{k-1,even}_{c}$) instead of $X^{k}_{c}$ and $Y^{k-1}_{c}$.

\end{enumerate}

\begin{rem}
We remark that the function inside the parentheses in the logarithm in $F_1$ and $F_2$ is uniformly bounded from below in $y$ for every $x$ by a strictly positive constant. Then we can analytically extend the integrand in $x$ to the strip $|\Im(z)| \leq c$ in such a way that the real part of this extension stays uniformly bounded away from 0 for a small enough $c$.
\end{rem}

\subsection{Step 1}

We start computing $k-1$ derivatives of $F(r,R)$ and showing that the terms are in $Y^{0}_{c}$. We have that

\begin{align*}
\pa^{k-1} F_{0}(r,R) & = \pa^{k-1} \left(\frac{r}{(1+r)^2}\left(\frac{r^2-1}{2}\sin(2x) + R'(x)(\cos^2(x)  + r \sin^{2}(x)) + (r-1)\sin(2x)R(x) + R(x)R'(x)\right)\right) \\
& = \frac{r}{(1+r)^2}\left(\pa^{k} R(x)(\cos^2(x)  + r \sin^{2}(x)) +  R(x)\pa^{k} R(x)\right) + \text{l.o.t}, \\
\end{align*}

which is clearly in $Y^{0}_{c}$ since $\|R(x\pm ic)\|_{L^{\infty}} \leq C\|R(x\pm ic)\|_{X^{k}_{c}}$ for $k \geq 3$. We move to $F_1$ and $F_2$. Throughout the computations, we will repeatedly use the following estimate. Let

\begin{align*}
A(x,y) & = \frac{(\cos(x)-\cos(x-y) + R(x)\cos(x) - R(x-y)\cos(x-y))^{2}}{4\left|\sin\left(\frac{y}{2}\right)\right|^{2}} \\
& + \frac{(r(\sin(x)-\sin(x-y))+ R(x)\sin(x) - R(x-y)\sin(x-y))^{2})}{4\left|\sin\left(\frac{y}{2}\right)\right|^{2}}.
\end{align*}

Using standard trigonometry, we can rewrite it as

\begin{align*}
A(x,y) & = \left(-\sin\left(x-\frac{y}{2}\right)(1+R(x)) + \frac{R(x)-R(x-y)}{2\sin\left(\frac{y}{2}\right)}\cos(x-y)\right)^{2} \\
& + \left(\cos\left(x-\frac{y}{2}\right)(r+R(x)) + \frac{R(x)-R(x-y)}{2\sin\left(\frac{y}{2}\right)}\sin(x-y)\right)^{2}.
\end{align*}

The goal is to obtain upper and lower bounds of $A(x \pm ic, y)$. We start with the upper bound:

\begin{align*}
\|A(x\pm ic, y)\|_{L^{\infty}} & \lesssim (1+\|R(x\pm ic)\|_{L^{\infty}})^{2} \cosh(2c) + \|R'(x \pm ic)\|_{L^{\infty}}^{2}\cosh(2c) \\
& + (r+\|R(x\pm ic)\|_{L^{\infty}})^{2} \cosh(2c) + \|R'(x \pm ic)\|_{L^{\infty}}^{2}\cosh(2c),
\end{align*}

where we have used that $|\cos(x\pm ic)|^{2}, |\cos(x\pm ic)|^{2} \leq \cosh(2c)$. Next, the following lower bounds follow:

\begin{align*}
|A(x\pm ic, y)| & \geq \left|1+R(x\pm ic)\right|\left|\sin\left(x\pm ic -\frac{y}{2}\right)\right|\left(\left|1+R(x\pm ic)\right|\left|\sin\left(x\pm ic -\frac{y}{2}\right)\right| - 2\|R'(x \pm ic)\|_{L^{\infty}}\sqrt{\cosh(2c)}\right) \\
& + \left|r+R(x\pm ic)\right|\left|\sin\left(x\pm ic -\frac{y}{2}\right)\right|\left(\left|1+R(x\pm ic)\right|\left|\sin\left(x\pm ic -\frac{y}{2}\right)\right| - 2\|R'(x \pm ic)\|_{L^{\infty}}\sqrt{\cosh(2c)}\right) \\
& \gtrsim r^{2},
\end{align*}

uniformly in $x$ and $y$ if $c$ is small enough. Using the previous bounds, it follows that

\begin{align}
\left\|\int \log(|A(x\pm ic,y)|) f(y) dy \right\|_{L^{2}} & \leq C \|\log(A(x\pm ic,y))\|_{L^{\infty}(x,y)} \|f\|_{L^{2}} \leq C \|f\|_{L^{2}} \nonumber \\
\left\|\int \log(|A(x\pm ic,y)|) f(x \pm ic -y) dy \right\|_{L^{2}} & \leq C \|f(x \pm ic)\|_{L^{2}} \label{boundkernelanalytic}
\end{align}

Taking $k-1$ derivatives of $F_1$, the most singular terms are:


\scriptsize
\begin{align*}
I_{1,1}(x) & = \int \log\left(A(x,y)\right)(\cos(x)(\pa^{k-1} R(x))+\sin(x)\pa^{k} R(x))(\sin(x-y) + R(x-y)\sin(x-y) - \cos(x-y)R'(x-y)) dy\\
I_{1,2}(x) & = \int \log\left(4\left|\sin\left(\frac{y}{2}\right)\right|^{2}\right)(\cos(x)(\pa^{k-1} R(x))+\sin(x)\pa^{k} R(x))(\sin(x-y) + R(x-y)\sin(x-y) - \cos(x-y)R'(x-y)) dy\\
I_{1,3}(x) & = \int \log\left(A(x,y)\right)(\cos(x)(r+R(x))+\sin(x)R'(x))(\pa^{k-1} R(x-y)\sin(x-y) - \cos(x-y)\pa^{k} R(x-y)) dy\\
I_{1,4}(x) & = \int \log\left(4\left|\sin\left(\frac{y}{2}\right)\right|^{2}\right)(\cos(x)(r+R(x))+\sin(x)R'(x))(\pa^{k-1} R(x-y)\sin(x-y) - \cos(x-y)\pa^{k} R(x-y)) dy\\
I_{1,5}(x) & = \int \frac{2(\cos(x)-\cos(x-y) + R(x)\cos(x) - R(x-y)\cos(x-y))( \pa^{k-1} R(x)\cos(x) - \pa^{k-1} R(x-y)\cos(x-y))}{\left((\cos(x)-\cos(x-y) + R(x)\cos(x) - R(x-y)\cos(x-y))^{2} + (r(\sin(x)-\sin(x-y)) + R(x)\sin(x) - R(x-y)\sin(x-y))^{2})\right)} \\
& \times \left[(\cos(x)(r+ R(x))+\sin(x) R'(x))(\sin(x-y) + R(x-y)\sin(x-y) - \cos(x-y)R'(x-y))\right. \\
& \left.- (\sin(x)+R(x)\sin(x)-\cos(x)R'(x))(\cos(x-y)(r+R(x-y))+\sin(x-y)R'(x-y))\right]dy\\
I_{1,6}(x) & = \int \frac{2(r(\sin(x)-\sin(x-y)) + R(x)\sin(x) - R(x-y)\sin(x-y))(\pa^{k-1} R(x)\sin(x) - \pa^{k-1} R(x-y)\sin(x-y))}{\left((\cos(x)-\cos(x-y) + R(x)\cos(x) - R(x-y)\cos(x-y))^{2} + (r(\sin(x)-\sin(x-y)) + R(x)\sin(x) - R(x-y)\sin(x-y))^{2})\right)}
 \\
& \times \left[(\cos(x)(r+ R(x))+\sin(x) R'(x))(\sin(x-y) + R(x-y)\sin(x-y) - \cos(x-y)R'(x-y))\right. \\
& \left.- (\sin(x)+R(x)\sin(x)-\cos(x)R'(x))(\cos(x-y)(r+R(x-y))+\sin(x-y)R'(x-y))\right]dy\\
\end{align*}
\normalsize

Using the bound \eqref{boundkernelanalytic}, we can show that

\begin{align*}
\|I_{1,1}(x \pm ic)\|_{L^{2}}, \|I_{1,3}(x \pm ic)\|_{L^{2}} < \infty,
\end{align*}

and

\begin{align*}
\|I_{1,2}(x \pm ic)\|_{L^{2}}, \|I_{1,4}(x \pm ic)\|_{L^{2}} < \infty,
\end{align*}

by virtue of the Generalized Young's inequality \cite{Folland:introduction-pdes}. We are only left to deal with $I_{1,5}(x)$ and $I_{1,6}(x)$. To do so, we will study the terms

\begin{align*}
B(x,y) & = \frac{2(\cos(x)-\cos(x-y) + R(x)\cos(x) - R(x-y)\cos(x-y))}{2\sin\left(\frac{y}{2}\right)} \\
& = \left(-2\sin\left(x-\frac{y}{2}\right)(1+R(x)) + 2\cos(x-y)\frac{R(x)-R(x-y)}{2\sin\left(\frac{y}{2}\right)}\right), \\
C(x,y) & = \frac{2(r\sin(x)-r\sin(x-y) + R(x)\sin(x) - R(x-y)\sin(x-y))}{2\sin\left(\frac{y}{2}\right)} \\
& = \left(2\cos\left(x-\frac{y}{2}\right)(r+R(x)) + 2\sin(x-y)\frac{R(x)-R(x-y)}{2\sin\left(\frac{y}{2}\right)}\right). \\
\end{align*}

It is immediate to see $B(x\pm ic,y), C(x\pm ic,y)$ are uniformly bounded in $L^{\infty}$ for a small enough $c$, and

\begin{align*}
D(x,y) & = \frac{1}{2\sin\left(\frac{y}{2}\right)}\left[(\cos(x)(r+ R(x))+\sin(x) R'(x))(\sin(x-y) + R(x-y)\sin(x-y) - \cos(x-y)R'(x-y))\right. \\
& \left.- (\sin(x)+R(x)\sin(x)-\cos(x)R'(x))(\cos(x-y)(r+R(x-y))+\sin(x-y)R'(x-y))\right]\\
& = D_{1}(x,y) + D_{2}(x,y) \\
D_{1}(x,y) & = \frac{(\cos(x)(r+ R(x))+\sin(x) R'(x))}{2\sin\left(\frac{y}{2}\right)} \\
& \times (\sin(x) - \sin(x-y) + R(x)\sin(x) - R(x-y)\sin(x-y) - \cos(x)R'(x)  + \cos(x-y)R'(x-y)) \\
& = (\cos(x)(r+ R(x))+\sin(x) R'(x)) \\
& \times \left(\cos\left(x-\frac{y}{2}\right)(1+R(x)) + \sin(x-y)\frac{R(x)-R(x-y)}{2\sin\left(\frac{y}{2}\right)} + \sin\left(x-\frac{y}{2}\right)R'(x) - \cos(x-y)\frac{R'(x)-R'(x-y)}{2\sin\left(\frac{y}{2}\right)}\right) \\
D_{2}(x,y) & = -\frac{(\sin(x)+R(x)\sin(x)-\cos(x)R'(x))}{2\sin\left(\frac{y}{2}\right)} \\
& \times (\cos(x)(r+R(x)) - \cos(x-y)(r+R(x-y)) + \sin(x)R'(x)  - \sin(x-y)R'(x-y)) \\
& = -(\sin(x)+R(x)\sin(x)-\cos(x)R'(x)) \\
& \times \left(-\sin\left(x-\frac{y}{2}\right)(r+R(x)) + \cos(x-y)\frac{R(x)-R(x-y)}{2\sin\left(\frac{y}{2}\right)} + R'(x)\cos\left(x-\frac{y}{2}\right) + \frac{R'(x)-R'(x-y)}{2\sin\left(\frac{y}{2}\right)}\sin(x-y)\right) \\
\end{align*}

Both $D_1(x \pm ic, y)$ and $D_{2}(x \pm ic, y)$ can therefore be bounded uniformly in $x$ and $y$ by a constant that depends on $c$ and $\|R\|_{X^{k}_{c}}$. Then, the bounds for $\|I_{1,5}(x\pm ic)\|_{L^{2}}$, $\|I_{1,6}(x \pm ic)\|_{L^{2}}$ follow easily using the boundedness of $A(x,y)$ and the representations

\begin{align*}
I_{1,5}(x) & = \int \frac{B(x,y)D(x,y)}{A(x,y)}(\pa^{k-1}R(x)\cos(x) - \pa^{k-1}R(x-y)\cos(x-y)) dy \\
I_{1,6}(x) & = \int \frac{C(x,y)D(x,y)}{A(x,y)}(\pa^{k-1}R(x)\sin(x) - \pa^{k-1}R(x-y)\sin(x-y)) dy.
\end{align*}

\subsection{Step 2}
\label{sectionpaso2elipses}
We note that this step was proved by Kirchhoff \cite{Kirchhoff:vorlesungen-math-physik}. Here we include it for completeness of the argument. Substituting $R(x) = 0$ in \eqref{funcionalelipses}, we obtain

\begin{align*}
 F(r,0) = \frac{r}{1+r} \frac{r-1}{2} \sin(2x) + \frac{1}{4\pi}\sum_{i=1}^{2} F_{i}(r,0),
\end{align*}

where

\begin{align*}
 F_{1}(r,0) & = -r\int_{0}^{2\pi} \log\left(2\sin^{2}\left(\frac{x-y}{2}\right)\right)(\sin(x-y))dy \\
F_{2}(r,0) & = -r\int_{0}^{2\pi} \log\left((1+r^2)+(r^2-1)\cos(x+y)\right)(\sin(x-y))dy
\end{align*}

We start computing $F_1(r,0)$, we get:

\begin{align*}
 F_{1}(r,0) & = -r\int_{0}^{2\pi} \log\left(2\sin^{2}\left(\frac{y}{2}\right)\right)(\sin(y))dy = 0
\end{align*}

Then

\begin{align*}
 F_{2}(r,0) & =  -r\int_{0}^{2\pi} \log\left((1+r^2)+(r^2-1)\cos(y)\right)(\sin(2x-y))dy \\
& =  -r\sin(2x)\int_{0}^{2\pi} \log\left((1+r^2)+(r^2-1)\cos(y)\right)\cos(y)dy \\
& = 2\pi r\sin(2x) \frac{1-r}{1+r}
\end{align*}

This implies that

\begin{align*}
F(r,0) & = \frac{r}{1+r} \frac{r-1}{2} \sin(2x) -  \frac{2\pi}{4\pi}\frac{r}{1+r} (r-1) \sin(2x) = 0,
\end{align*}

as we wanted to prove.

\subsection{Step 3}

We need to prove the existence and the continuity of the Gateaux derivatives $\pa_{R} F(r,R)$, $\pa_{r} F(r,R)$ and $\pa_{r,R}F(r,R)$. We have the following Lemma:



\begin{lemma}
\label{lemagatoderivada}
 For all $R\in V^r$ and for all $h\in X_{c}^{k}$ such that $||h||_{X^{k}_{c}}=1$ we have that
$$\lim_{t \to 0}\frac{F_i(R+th)-F_i(R)}{t}= D_i[R]h\quad \text{in $X^{k-1}_{c}$},$$
where
\small
\begin{align*}
D_{0}[R] h & = \frac{r}{(1+r)^2}\left( h'(x)(\cos^2(x)  + r \sin^{2}(x)) + (r-1)\sin(2x)h(x) + h(x)R'(x) + R(x)h'(x)\right) \\
D_1[R] h & = \int \log(A(x,y))(\cos(x)h(x))+\sin(x)h'(x))(\sin(x-y) + R(x-y)\sin(x-y) - \cos(x-y)R'(x-y)) dy\\
& + \int \log\left(4\left|\sin\left(\frac{y}{2}\right)\right|^{2}\right)(\cos(x)h(x))+\sin(x)h'(x))(\sin(x-y) + R(x-y)\sin(x-y) - \cos(x-y)R'(x-y)) dy\\
& + \int \log(A(x,y))(\cos(x)(r+R(x))+\sin(x)R'(x))(h(x-y)\sin(x-y) - \cos(x-y)h'(x-y)) dy\\
& + \int \log\left(4\left|\sin\left(\frac{y}{2}\right)\right|^{2}\right)(\cos(x)(r+R(x))+\sin(x)R'(x))(h(x-y)\sin(x-y) - \cos(x-y)h'(x-y)) dy\\
& + \int \frac{2(\cos(x)-\cos(x-y) + R(x)\cos(x) - R(x-y)\cos(x-y))(h(x)\cos(x) - h(x-y)\cos(x-y))}{(\cos(x)-\cos(x-y) + R(x)\cos(x) - R(x-y)\cos(x-y))^{2} + (r(\sin(x)-\sin(x-y)) + R(x)\sin(x) - R(x-y)\sin(x-y))^{2}}
 \\
& \times (\cos(x)(r+R(x))+\sin(x)R'(x))(\sin(x-y) + R(x-y)\sin(x-y) - \cos(x-y)R'(x-y)) dy\\
 & + \int \frac{2(r(\sin(x)-\sin(x-y)) + R(x)\sin(x) - R(x-y)\sin(x-y))(h(x)\sin(x) - h(x-y)\sin(x-y))}{(\cos(x)-\cos(x-y) + R(x)\cos(x) - R(x-y)\cos(x-y))^{2} + (r(\sin(x)-\sin(x-y)) + R(x)\sin(x) - R(x-y)\sin(x-y))^{2}}
 \\
& \times (\cos(x)(r+R(x))+\sin(x)R'(x))(\sin(x-y) + R(x-y)\sin(x-y) - \cos(x-y)R'(x-y)) dy\\
D_{2}[R] h & = -\int \log(A(x,y))(\sin(x)+R(x)\sin(x)-\cos(x)R'(x))(\cos(x-y)h(x-y)+\sin(x-y)h'(x-y)) dy\\
& -\int \log\left(4\left|\sin\left(\frac{y}{2}\right)\right|^{2}\right)(\sin(x)+R(x)\sin(x)-\cos(x)R'(x))(\cos(x-y)h(x-y)+\sin(x-y)h'(x-y)) dy\\
 & - \int \log(A(x,y))(h(x)\sin(x)-\cos(x)h'(x))(\cos(x-y)(r+R(x-y))+\sin(x-y)R'(x-y)) dy\\
 & - \int \log\left(4\left|\sin\left(\frac{y}{2}\right)\right|^{2}\right)(h(x)\sin(x)-\cos(x)h'(x))(\cos(x-y)(r+R(x-y))+\sin(x-y)R'(x-y)) dy\\
& - \int \frac{2(\cos(x)-\cos(x-y) + R(x)\cos(x) - R(x-y)\cos(x-y))(h(x)\cos(x) - h(x-y)\cos(x-y))}{(\cos(x)-\cos(x-y) + R(x)\cos(x) - R(x-y)\cos(x-y))^{2} + (r(\sin(x)-\sin(x-y)) + R(x)\sin(x) - R(x-y)\sin(x-y))^{2}}
 \\
& \times (\sin(x)+R(x)\sin(x)-\cos(x)R'(x))(\cos(x-y)(r+R(x-y))+\sin(x-y)R'(x-y)) dy\\
& - \int \frac{2(r(\sin(x)-\sin(x-y)) + R(x)\sin(x) - R(x-y)\sin(x-y))(h(x)\sin(x) - h(x-y)\sin(x-y))}{(\cos(x)-\cos(x-y) + R(x)\cos(x) - R(x-y)\cos(x-y))^{2} + (r(\sin(x)-\sin(x-y)) + R(x)\sin(x) - R(x-y)\sin(x-y))^{2}}
 \\
& \times (\sin(x)+R(x)\sin(x)-\cos(x)R'(x))(\cos(x-y)(r+R(x-y))+\sin(x-y)R'(x-y)) dy\\
\end{align*}
\normalsize
Moreover, $D_i[R] h$ are continuous in $R$.
\end{lemma}

\begin{proof}
Straightforward computation.
\end{proof}

\subsection{Step 4}

Before starting Step 4, we compute the linearization of $F$ around 0 in the direction $h(x)$. To do so, we will compute the contribution of each frequency separately. Assume $h(x) = \cos(kx), k \geq 2$. We will highlight the differences in some of the terms for the case $k = 1$ later. By taking $R = 0$ in Lemma \ref{lemagatoderivada} one sees that this linearization is equal to

\begin{align*}
\left(\frac{1}{4}(2-k)(1-r)\sin((k-2)x)-\frac{k}{2}(1+r)\sin(kx)-\frac{1}{4}(2+k)(1-r)\sin((k+2)x)\right)\frac{r}{(1+r)^2} + \frac{1}{4\pi}\sum_{i=1}^{5} I_{i},
\end{align*}

where:

\begin{align*}
I_{1} & = -r\int_{0}^{2\pi}\frac{\sin(x-y)2(\cos(x)-\cos(y))(\cos(x)h(x)-\cos(y)h(y))}{2\sin^{2}\left(\frac{x-y}{2}\right)((1+r^2)+(r^2-1)\cos(x+y))}dy \\
I_{2} & = -r^2\int_{0}^{2\pi}\frac{\sin(x-y)2(\sin(x)-\sin(y))(h(x)\sin(x)-h(y)\sin(y))}{2\sin^{2}\left(\frac{x-y}{2}\right)((1+r^2)+(r^2-1)\cos(x+y))}dy \\
I_{3} & = \int_{0}^{2\pi}\log\left(\sin^{2}\left(\frac{x-y}{2}\right)((1+r^2)+(r^2-1)\cos(x+y))\right)h(x)(-r\cos(y)\sin(x)+\cos(x)\sin(y))dy \\
I_{4} & = \int_{0}^{2\pi}\log\left(\sin^{2}\left(\frac{x-y}{2}\right)((1+r^2)+(r^2-1)\cos(x+y))\right)h(y)(-\cos(y)\sin(x)+r\cos(x)\sin(y))dy \\
I_{5} & = \int_{0}^{2\pi}\log\left(\sin^{2}\left(\frac{x-y}{2}\right)((1+r^2)+(r^2-1)\cos(x+y))\right)(h'(x)-h'(y))(r\cos(x)\cos(y)+\sin(x)\sin(y))dy \\
\end{align*}

Then, one can show the following Lemma:

\begin{lemma}

Let

\begin{align*}
 h(x) = \sum_{k=1}^{\infty} a_{k} \cos(kx),
\end{align*}

then $DF(0,r)[h]$ is given by

\begin{align*}
 \sum_{k=1}^{\infty} (x_{k}a_{k-2} + y_{k}a_{k} + z_{k}a_{k+2}) \sin(kx),
\end{align*}

where, for $k > 0$:

\begin{align*}
 x_{k} = -\frac{(1-r)}{8(1+r)^{2}} \left(-1 - 2 r + 2 k r - r^2 - \frac{(1-r)^{k}}{(1+r)^{k-2}}\right) \\
y_{1} = -\frac{1}{4(1+r)^2}(3r+1), \quad y_{k} =  \frac{1}{4(1+r)}\left(-1 - 2 r + 2 k r - r^2 - \frac{(1-r)^{k}}{(1+r)^{k-2}}\right) \quad k > 1, \\
z_{k} =  -\frac{1-r}{8(r+1)^2}\left(-1 - 2 r + 2 k r - r^2 - \frac{(1-r)^{k}}{(1+r)^{k-2}}\right),\\
\end{align*} 
and the convention that $a_k = 0$ for $k \leq 0$.
\end{lemma}

\begin{proof}

By linearity, we can calculate with $h(x) = \cos(kx)$ and sum over all contributions. We start with $I_1$. Using that

\begin{align*}
 \sin(x-y) & = 2\sin\left(\frac{x-y}{2}\right)\cos\left(\frac{x-y}{2}\right) \\
\cos(x)-\cos(y) & = (-2)\sin\left(\frac{x-y}{2}\right)\sin\left(\frac{x+y}{2}\right),
\end{align*}

$I_1$ reads:

\begin{align*}
 I_{1} & = 4r\int_{0}^{2\pi}\frac{\cos\left(\frac{x-y}{2}\right)\sin\left(\frac{x+y}{2}\right)(\cos(x)h(x)-\cos(y)h(y))}{((1+r^2)+(r^2-1)\cos(x+y))}dy \\
& = 4r\int_{0}^{2\pi}\frac{\cos\left(\frac{2x-y}{2}\right)\sin\left(\frac{y}{2}\right)(\cos(x)h(x)-\cos(y-x)h(y-x))}{((1+r^2)+(r^2-1)\cos(y))}dy \\
\end{align*}

and we split it into $I_{11} + I_{12}$, where:

\begin{align*}
 I_{11} & = 2r\int_{0}^{2\pi}\frac{\cos(x)\sin(y)(\cos(x)h(x)-\cos(y-x)h(y-x))}{((1+r^2)+(r^2-1)\cos(y))}dy \\
 I_{12} & = 2r\int_{0}^{2\pi}\frac{\sin(x)(1-\cos(y))(\cos(x)h(x)-\cos(y-x)h(y-x))}{((1+r^2)+(r^2-1)\cos(y))}dy \\
\end{align*}

Then, $I_{11} = I_{111} + I_{112}$:

\begin{align*}
 I_{111} & = 2r\int_{0}^{2\pi}\frac{\cos(x)\sin(y)\cos(x)h(x)}{((1+r^2)+(r^2-1)\cos(y))}dy = 0 \\
I_{112} & = -2r\int_{0}^{2\pi}\frac{\cos(x)\sin(y)\cos(y-x)h(y-x)}{((1+r^2)+(r^2-1)\cos(y))}dy \\
& = -r\int_{0}^{2\pi}\frac{\cos(x)\sin(y)\sin((k+1)x)\sin((k+1)y)}{((1+r^2)+(r^2-1)\cos(y))}dy \\
&  -r\int_{0}^{2\pi}\frac{\cos(x)\sin(y)\sin((k-1)x)\sin((k-1)y)}{((1+r^2)+(r^2-1)\cos(y))}dy \\
& = \frac{r}{2}\cos(x)\sin((k+1)x)\int_{0}^{2\pi}\frac{\cos((k+2)y)}{((1+r^2)+(r^2-1)\cos(y))}dy \\
& - \frac{r}{2}\cos(x)\sin((k+1)x)\int_{0}^{2\pi}\frac{\cos(ky)}{((1+r^2)+(r^2-1)\cos(y))}dy \\
& + \frac{r}{2}\cos(x)\sin((k-1)x)\int_{0}^{2\pi}\frac{\cos(ky)}{((1+r^2)+(r^2-1)\cos(y))}dy \\
& - \frac{r}{2}\cos(x)\sin((k-1)x)\int_{0}^{2\pi}\frac{\cos((k-2)y)}{((1+r^2)+(r^2-1)\cos(y))}dy \\
& = \frac{r}{2}\cos(x)\sin((k+1)x)\frac{\pi}{r}\left(\frac{1-r}{1+r}\right)^{k+2} - \frac{r}{2}\cos(x)\sin((k+1)x)\frac{\pi}{r}\left(\frac{1-r}{1+r}\right)^{k} \\
& + \frac{r}{2}\cos(x)\sin((k-1)x)\frac{\pi}{r}\left(\frac{1-r}{1+r}\right)^{k}  - \frac{r}{2}\cos(x)\sin((k-1)x)\frac{\pi}{r}\left(\frac{1-r}{1+r}\right)^{k-2}
\end{align*}

We now split $I_{12} = I_{121} + I_{122}$:

\begin{align*}
 I_{121} & = 2r\int_{0}^{2\pi}\frac{\sin(x)(1-\cos(y))\cos(x)\cos(kx)}{((1+r^2)+(r^2-1)\cos(y))}dy \\
& = 2r\cos(x)\cos(kx)\sin(x)\int_{0}^{2\pi}\frac{(1-\cos(y))}{((1+r^2)+(r^2-1)\cos(y))}dy \\
& = 2r\cos(x)\cos(kx)\sin(x)\frac{\pi}{r}\frac{2r}{1+r} \\
 I_{122} & = -2r\int_{0}^{2\pi}\frac{\sin(x)(1-\cos(y))\cos(y-x)\cos(k(y-x))}{((1+r^2)+(r^2-1)\cos(y))}dy \\
& = -r\sin(x)\int_{0}^{2\pi}\frac{(1-\cos(y))\cos((k+1)(y-x))}{((1+r^2)+(r^2-1)\cos(y))}dy \\
& -r\sin(x)\int_{0}^{2\pi}\frac{(1-\cos(y))\cos((k-1)(y-x))}{((1+r^2)+(r^2-1)\cos(y))}dy \\
& = -r\sin(x)\cos((k+1)x)\int_{0}^{2\pi}\frac{(1-\cos(y))\cos((k+1)y)}{((1+r^2)+(r^2-1)\cos(y))}dy \\
& -r\sin(x)\cos((k-1)x)\int_{0}^{2\pi}\frac{(1-\cos(y))\cos((k-1)y)}{((1+r^2)+(r^2-1)\cos(y))}dy \\
& = 2r\sin(x)\cos((k+1)x)\frac{\pi}{r}\left(\frac{1-r}{1+r}\right)^{k}\left(\frac{r}{r+1}\right)^{2} \\
& +2r\sin(x)\cos((k-1)x)\frac{\pi}{r}\left(\frac{1-r}{1+r}\right)^{k-2}\left(\frac{r}{r+1}\right)^{2} \\
\end{align*}

We move on to $I_2$. Using that

\begin{align*}
 \sin(x-y) & = 2\sin\left(\frac{x-y}{2}\right)\cos\left(\frac{x-y}{2}\right) \\
\sin(x)-\sin(y) & = 2\sin\left(\frac{x-y}{2}\right)\cos\left(\frac{x+y}{2}\right),
\end{align*}

$I_2$ reads:

\begin{align*}
I_{2} & = -4r^2\int_{0}^{2\pi}\frac{\cos\left(\frac{x-y}{2}\right)\cos\left(\frac{x+y}{2}\right)(h(x)\sin(x)-h(y)\sin(y))}{((1+r^2)+(r^2-1)\cos(x+y))}dy \\
& = -2r^2\int_{0}^{2\pi}\frac{(\cos(x)+\cos(y))(h(x)\sin(x)-h(y)\sin(y))}{((1+r^2)+(r^2-1)\cos(x+y))}dy \\
\end{align*}

We split $I_2$ into $I_{21} + I_{22} + I_{23} + I_{24}$, where:

\begin{align*}
 I_{21} & = -2r^2\int_{0}^{2\pi}\frac{\cos(x)h(x)\sin(x)}{((1+r^2)+(r^2-1)\cos(x+y))}dy \\
I_{22} & = 2r^2\int_{0}^{2\pi}\frac{\cos(x)h(y)\sin(y)}{((1+r^2)+(r^2-1)\cos(x+y))}dy \\
I_{23} & = -2r^2\int_{0}^{2\pi}\frac{\cos(y)h(x)\sin(x)}{((1+r^2)+(r^2-1)\cos(x+y))}dy \\
I_{24} & = 2r^2\int_{0}^{2\pi}\frac{\cos(y)h(y)\sin(y)}{((1+r^2)+(r^2-1)\cos(x+y))}dy \\
\end{align*}

We obtain, on the one hand:

\begin{align*}
 I_{21} & = -2r^2\cos(kx)\cos(x)\sin(x)\frac{\pi}{r} \\
I_{22} & = r^2\cos(x)\int_{0}^{2\pi}\frac{\sin((k+1)(y-x))}{((1+r^2)+(r^2-1)\cos(y))}dy \\
&  -r^2\cos(x)\int_{0}^{2\pi}\frac{\sin((k-1)(y-x))}{((1+r^2)+(r^2-1)\cos(y))}dy \\
& = -r^2\sin((k+1)x)\cos(x)\int_{0}^{2\pi}\frac{\cos((k+1)y)}{((1+r^2)+(r^2-1)\cos(y))}dy \\
&  +r^2\cos(x)\sin((k-1)x)\int_{0}^{2\pi}\frac{\cos((k-1)y))}{((1+r^2)+(r^2-1)\cos(y))}dy \\
& = -r^2\sin((k+1)x)\cos(x)\frac{\pi}{r}\left(\frac{1-r}{1+r}\right)^{k+1} \\
&  +r^2\cos(x)\sin((k-1)x)\frac{\pi}{r}\left(\frac{1-r}{1+r}\right)^{k-1}
\end{align*}

On the other:

\begin{align*}
I_{23} & = -2r^2\cos(kx)\sin(x)\cos(x)\frac{\pi}{r}\left(\frac{1-r}{1+r}\right) \\
I_{24} & = \frac12 r^2\int_{0}^{2\pi}\frac{\sin((k+2)(y-x))}{((1+r^2)+(r^2-1)\cos(y))}dy \\
& -\frac12 r^2\int_{0}^{2\pi}\frac{\sin((k-2)(y-x))}{((1+r^2)+(r^2-1)\cos(y))}dy \\
& = -\frac12 r^2\sin((k+2)x)\frac{\pi}{r}\left(\frac{1-r}{1+r}\right)^{k+2} \\
& +\frac12 r^2\sin((k-2)x)\frac{\pi}{r}\left(\frac{1-r}{1+r}\right)^{k-2} \\
\end{align*}

The next term is $I_3$. We start splitting into $I_{31} + I_{32}$, where:

\begin{align*}
I_{31} & = \int_{0}^{2\pi}\log\left(\sin^{2}\left(\frac{x-y}{2}\right)\right)h(x)(-r\cos(y)\sin(x)+\cos(x)\sin(y))dy \\
I_{32} & = \int_{0}^{2\pi}\log\left((1+r^2)+(r^2-1)\cos(x+y))\right)h(x)(-r\cos(y)\sin(x)+\cos(x)\sin(y))dy \\
\end{align*}

We further split $I_{31}$ into:

\begin{align*}
 I_{311} & = -r\cos(kx)\sin(x)\int_{0}^{2\pi}\log\left(\sin^{2}\left(\frac{y}{2}\right)\right)\cos(x-y)dy \\
  & = r\cos(kx)\sin(x)\cos(x)(2\pi) \\
I_{312} & = \cos(kx)\cos(x)\int_{0}^{2\pi}\log\left(\sin^{2}\left(\frac{y}{2}\right)\right)\sin(x-y)dy \\
& = \cos(kx)\cos(x)\sin(x)(-2\pi) \\
\end{align*}

and $I_{32}$ into:

\begin{align*}
I_{321} & = -r\cos(kx)\sin(x)\int_{0}^{2\pi}\log\left((1+r^2)+(r^2-1)\cos(y))\right)\cos(y-x)dy \\
& = -r\cos(kx)\sin(x)\cos(x)(-2\pi)\left(\frac{1-r}{1+r}\right) \\
I_{322} & = \cos(kx)\cos(x)\int_{0}^{2\pi}\log\left((1+r^2)+(r^2-1)\cos(y))\right)\sin(y-x)dy \\
& = -\sin(x)\cos(kx)\cos(x)(-2\pi)\left(\frac{1-r}{1+r}\right) \\
\end{align*}

This finishes $I_3$. The next term is $I_4 = I_{41} + I_{42}$:

\begin{align*}
I_{41} & = \int_{0}^{2\pi}\log\left(\sin^{2}\left(\frac{x-y}{2}\right)\right)h(y)(-\cos(y)\sin(x)+r\cos(x)\sin(y))dy \\
I_{42} & = \int_{0}^{2\pi}\log\left(((1+r^2)+(r^2-1)\cos(x+y))\right)h(y)(-\cos(y)\sin(x)+r\cos(x)\sin(y))dy \\
\end{align*}

We start computing $I_{41} = I_{411} + I_{412}$:

\begin{align*}
 I_{411} & = -\frac12\sin(x)\int_{0}^{2\pi}\log\left(\sin^{2}\left(\frac{y}{2}\right)\right)\cos((k+1)(x-y))dy \\
& -\frac12\sin(x)\int_{0}^{2\pi}\log\left(\sin^{2}\left(\frac{y}{2}\right)\right)\cos((k-1)(x-y))dy \\
 & = -\frac12\sin(x)\cos((k+1)x)(-2\pi)\frac{1}{k+1} \\
& -\frac12\sin(x)\cos((k-1)x)(-2\pi)\frac{1}{k-1} \\
I_{412} & = \frac12 r\cos(x)\int_{0}^{2\pi}\log\left(\sin^{2}\left(\frac{y}{2}\right)\right)\sin((k+1)(x-y))dy \\
& -\frac12 r\cos(x)\int_{0}^{2\pi}\log\left(\sin^{2}\left(\frac{y}{2}\right)\right)\sin((k-1)(x-y))dy \\
& = -\frac12 r\cos(x)\sin((k+1)x)(2\pi)\frac{1}{k+1}\\
& +\frac12 r\cos(x)\sin((k-1)x)(2\pi)\frac{1}{k-1}
\end{align*}

We now split $I_{42}$ into $I_{421} + I_{422}$, where:

\begin{align*}
I_{421} & = -\frac12\sin(x)\int_{0}^{2\pi}\log\left(((1+r^2)+(r^2-1)\cos(y))\right)\cos((k+1)(y-x))dy \\
 & -\frac12\sin(x)\int_{0}^{2\pi}\log\left(((1+r^2)+(r^2-1)\cos(y))\right)\cos((k-1)(y-x))dy \\ 
& = \frac12\sin(x)\cos((k+1)x)\frac{2\pi}{k+1}\left(\frac{1-r}{r+1}\right)^{k+1} \\
 & +\frac12\sin(x)\cos((k-1)x)\frac{2\pi}{k-1}\left(\frac{1-r}{r+1}\right)^{k-1} \\
I_{422} & = \frac12 r\cos(x)\int_{0}^{2\pi}\log\left(((1+r^2)+(r^2-1)\cos(y))\right)\sin((k+1)(y-x))dy \\
 &  - \frac12 r\cos(x)\int_{0}^{2\pi}\log\left(((1+r^2)+(r^2-1)\cos(y))\right)\sin((k-1)(y-x))dy \\
& = \frac12 r\cos(x)\sin((k+1)x)\frac{2\pi}{k+1}\left(\frac{1-r}{r+1}\right)^{k+1} \\
 &  - \frac12 r\cos(x)\sin((k-1)x)\frac{2\pi}{k-1}\left(\frac{1-r}{r+1}\right)^{k-1}.
\end{align*}

This concludes with $I_{42}$ and therefore $I_{4}$. Finally, we move on to $I_{5} = I_{51} + I_{52}$, with:

\begin{align*}
I_{51} & = \int_{0}^{2\pi}\log\left(\sin^{2}\left(\frac{x-y}{2}\right))\right)(h'(x)-h'(y))(r\cos(x)\cos(y)+\sin(x)\sin(y))dy \\ 
I_{52} & = \int_{0}^{2\pi}\log\left(((1+r^2)+(r^2-1)\cos(x+y))\right)(h'(x)-h'(y))(r\cos(x)\cos(y)+\sin(x)\sin(y))dy \\
\end{align*}

We now decompose $I_{51}$ as $I_{511} + I_{512} + I_{513} + I_{514}$, giving:

\begin{align*}
I_{511} & = -kr\sin(kx)\cos(x)\int_{0}^{2\pi}\log\left(\sin^{2}\left(\frac{y}{2}\right))\right)\cos(x-y)dy \\ 
& = kr\sin(kx)\cos(x)\cos(x)(2\pi) \\
I_{512} & = -k\sin(kx)\sin(x)\int_{0}^{2\pi}\log\left(\sin^{2}\left(\frac{y}{2}\right))\right)\sin(x-y)dy \\ 
& = k\sin(kx)\sin(x)\sin(x)(2\pi) \\
I_{513} & = \frac12 rk\cos(x)\int_{0}^{2\pi}\log\left(\sin^{2}\left(\frac{y}{2}\right))\right)\sin((k+1)(x-y))dy \\ 
& + \frac12 rk\cos(x)\int_{0}^{2\pi}\log\left(\sin^{2}\left(\frac{y}{2}\right))\right)\sin((k-1)(x-y))dy \\ 
& = -\frac12 rk\cos(x)\sin((k+1)x)\frac{2\pi}{k+1} \\
& - \frac12 rk\cos(x)\sin((k-1)x)\frac{2\pi}{k-1} \\
I_{514} & = \frac12 k\sin(x)\int_{0}^{2\pi}\log\left(\sin^{2}\left(\frac{y}{2}\right))\right)\cos((k-1)(x-y))dy \\ 
& - \frac12 k\sin(x)\int_{0}^{2\pi}\log\left(\sin^{2}\left(\frac{y}{2}\right))\right)\cos((k+1)(x-y))dy \\ 
& = -\frac12 k\sin(x)\cos((k-1)x)\frac{2\pi}{k-1} \\
& + \frac12 k\sin(x)\cos((k+1)x)\frac{2\pi}{k+1}
\end{align*}

The last term is $I_{52} = I_{521} + I_{522} + I_{523} + I_{524}$:

\begin{align*}
I_{521} & = -kr\sin(kx)\cos(x)\int_{0}^{2\pi}\log\left(((1+r^2)+(r^2-1)\cos(y))\right)\cos(y-x)dy \\
& = kr\sin(kx)\cos(x)\cos(x)(2\pi)\left(\frac{1-r}{1+r}\right) \\
I_{522} & = -k\sin(kx)\sin(x)\int_{0}^{2\pi}\log\left(((1+r^2)+(r^2-1)\cos(y))\right)\sin(y-x)dy \\
& = -k\sin(kx)\sin(x)\sin(x)(2\pi)\left(\frac{1-r}{1+r}\right) \\
I_{523} & = \frac12 kr\cos(x)\int_{0}^{2\pi}\log\left(((1+r^2)+(r^2-1)\cos(y))\right)\sin((k+1)(y-x))dy \\
& + \frac12 kr\cos(x)\int_{0}^{2\pi}\log\left(((1+r^2)+(r^2-1)\cos(y))\right)\sin((k-1)(y-x))dy \\
& = \frac12 kr\cos(x)\sin((k+1)x)\frac{2\pi}{k+1}\left(\frac{1-r}{1+r}\right)^{k+1} \\
& + \frac12 kr\cos(x)\sin((k-1)x)\frac{2\pi}{k-1}\left(\frac{1-r}{1+r}\right)^{k-1} \\
I_{524} & = -\frac12 k\sin(x)\int_{0}^{2\pi}\log\left(((1+r^2)+(r^2-1)\cos(y))\right)\cos((k+1)(y-x))dy \\
& +\frac12 k\sin(x)\int_{0}^{2\pi}\log\left(((1+r^2)+(r^2-1)\cos(y))\right)\cos((k-1)(y-x))dy \\
& = \frac12 k\sin(x)\cos((k+1)x)\frac{2\pi}{k+1}\left(\frac{1-r}{1+r}\right)^{k+1} \\
& -\frac12 k\sin(x)\cos((k-1)x)\frac{2\pi}{k-1}\left(\frac{1-r}{1+r}\right)^{k-1}.
\end{align*}

In the case $k = 1$, most of the terms are equal, other than the following, that yield:

\begin{align*}
 I_{122} & = 2r\sin(x)\cos((k+1)x)\frac{\pi}{r}\left(\frac{1-r}{1+r}\right)^{k}\left(\frac{r}{r+1}\right)^{2} \\
& -2\sin(x)\frac{\pi r}{1+r} \\
I_{24} & = -\frac12 r^2\sin((k+2)x)\frac{\pi}{r}\left(\frac{1-r}{1+r}\right)^{k+2} +\frac12 r^2\sin((k-2)x)\frac{\pi}{r}\left(\frac{1-r}{1+r}\right) \\
I_{411} & = -\frac12\sin(x)\cos((k+1)x)(-2\pi)\frac{1}{k+1} -\frac12\sin(x)(-4\pi \log(2)) \\
I_{412} & = -\frac12 r\cos(x)\sin((k+1)x)(2\pi)\frac{1}{k+1}\\
I_{421} & = \frac12\sin(x)\cos((k+1)x)\frac{2\pi}{k+1}\left(\frac{1-r}{r+1}\right)^{k+1} -\frac12\sin(x)\left(2\pi \log\left(\frac{(1+r)^2}{2}\right)\right) \\
I_{422} & = \frac12 r\cos(x)\sin((k+1)x)\frac{2\pi}{k+1}\left(\frac{1-r}{r+1}\right)^{k+1} \\
I_{513} & = -\frac12 rk\cos(x)\sin((k+1)x)\frac{2\pi}{k+1} \\
I_{514} & = \frac12 k\sin(x)(-4\pi \log(2)) + \frac12 k\sin(x)\cos((k+1)x)\frac{2\pi}{k+1} \\
I_{523} & = \frac12 kr\cos(x)\sin((k+1)x)\frac{2\pi}{k+1}\left(\frac{1-r}{1+r}\right)^{k+1} \\
I_{524} & = \frac12 k\sin(x)\cos((k+1)x)\frac{2\pi}{k+1}\left(\frac{1-r}{1+r}\right)^{k+1} +\frac12 k\sin(x)\left(2\pi \log\left(\frac{(1+r)^2}{2}\right)\right),
\end{align*}

where the calculations are carried out the same way as in the case $k>1$. Adding up all the contributions, we get the desired result.

\end{proof}

We can prove the following:

\begin{lemma}
\label{lemaraices}
For every $k > 2$:
\begin{itemize}
 \item There exists a unique $0 < r(k) < 1$ such that all of $x_{k}, y_{k}, z_{k}$ are equal to zero. 
\item The root $r(k)$ is a simple root.
\item Moreover, $x_{j}, y_{j}, z_{j}$ will be nonzero for all $j \neq k$.
\item If $k_1 > k_2$, then $r(k_1) < r(k_2)$.
\end{itemize}
\end{lemma}

\begin{proof}
 It is clear that the restriction in the size of $r$ reduces the problem to study the bracket 

\begin{align*}
 \left(-1 - 2 r + 2 k r - r^2 - \frac{(1-r)^{k}}{(1+r)^{k-2}}\right)
\end{align*}

We first do the change of variables $r = \frac{z-1}{z+1}$. The bracket reads:

\begin{align*}
 B(z) = \frac{2}{(1+z)^2}\left((2-k)z^2 + k + 2 z^{2-k}\right),
\end{align*}

where the range of study is $1 < z < \infty$. Moreover, the change of variables is a bijection so it is enough to show it for the bracket written in the $z$ variables.

The existence of a solution is given by Bolzano's theorem since $B(1) = 2, \lim_{z \to \infty} B(z) = 4-2k$. The uniqueness follows from Descartes' rule of signs applied to $z^{k-2}(1+z)^2B(z)$: there is only a change of signs so there is only a positive root (which has to be the one found before). This also shows that the root is simple.

Finally, $x_{j}, y_{j}, z_{j}$ are clearly non-zero for $j \neq k$ since none of the factors other than the bracket could produce a zero at $r = r(k)$.

To conclude, we observe that for every $0 < r < 1$:

\begin{align*}
\pa_{k} \left(-1 - 2 r + 2 k r - r^2 - \frac{(1-r)^{k}}{(1+r)^{k-2}}\right)
= 2r - 2\left(\frac{1-r}{1+r}\right)^{k} \log\left(\frac{1-r}{1+r}\right) > 0,
\end{align*}

from which the monotonicity of the roots follows.

\end{proof}

\begin{rem}
 Some values of $r(k)$ were calculated numerically for $k = 4, k = 5$ in \cite{Kamm:thesis-shape-stability-patches}. 
\end{rem}

\begin{defi}
 For every $k$, let $K_{k}(r) = -\frac{1-r}{8(r+1)^2}\left(-1 - 2 r + 2 k r - r^2 - \frac{(1-r)^{k}}{(1+r)^{k-2}}\right)$ so that

\begin{align*}
 x_{k} = K_{k}(r), \quad y_{k} = -K_{k}(r)2\frac{1+r}{1-r} = -2 K_{k}(r) z, \quad z_{k} = K_{k}(r).
\end{align*}

\end{defi}

\begin{lemma}
\label{lemanucleo}

Let $k>2$. Then Ker$(DF)(r(k),0)[h]$ is one dimensional.

\end{lemma}

\begin{proof}
 We will first show the existence of a nontrivial kernel, calculating a non-zero generator, then show uniqueness (up to scalar multiples) of the generator. Let $z = z(k) = \frac{1+r(k)}{1-r(k)}$.

We will discuss the case where $k$ is even, the odd case being proved similarly.

Let $\lambda_{+} = z + \sqrt{z^2-1}, \lambda_{-} = z - \sqrt{z^{2}-1}$. We remark that $\lambda_{+} > 1 > \lambda_{-}$. We will define $h_{0}$ (the generator of the kernel) in the following way:

\begin{align*}
 h_{0}(x) = \sum_{p=1}^{\infty}c_{p}\cos(2px),
\end{align*}

where
\begin{align*}
 c_{p} = \left\{
\begin{array}{cc}
 \lambda_{+}^{p} - \lambda_{-}^{p} & \text{ if } 1 \leq p \leq k \\
 (\lambda_{+}^{k} - \lambda_{-}^{k})\lambda_{-}^{p-k} & \text{ if } k \leq p < \infty
\end{array}
\right\}
\end{align*}

Then it is immediate to check that

\begin{align*}
 c_{p-1} - 2z c_{p} + c_{p+1} & = 0,
\end{align*}

for every $p \neq k$, with the convention that $c_{0} = 0$, since for $p < k$:

\begin{align*}
 c_{p-1} - 2z c_{p} + c_{p+1} & = \lambda_{+}^{p-1}(1+\lambda_{+}^{2} - 2z\lambda_{+}) - \lambda_{-}^{p-1}(1+\lambda_{-}^{2} - 2z\lambda_{-}) = 0,
\end{align*}

and for $p > k$:

\begin{align*}
 c_{p-1} - 2z c_{p} + c_{p+1} & = (\lambda_{+}^{k} - \lambda_{-}^{k})\lambda_{-}^{p-1-k}(\lambda_{-}^{2} + 1 - 2z \lambda_{-}) = 0.
\end{align*}

We remark that $h_0 \in X_{c}^{k}$ for a small enough $c$ since the coefficients $c_{p}$ decay exponentially fast. Next, we will see that if there are other choices of the coefficients $c_{p}$, they represent functions which are not in $X_{c}^{k}$. We will argue by contradiction. Suppose that there exists a function $h_{0}'$ which is represented by the coefficients $\{c_{p}'\}_{p=1}^{\infty}$. We distinguish two cases:

If $c_{k}' = 0$, this implies that $c_{p}' = 0$ for $p < k$. Moreover, let $c_{k+1}' = c$. If $c = 0$, then $c_{p}' = 0$ for all $p > k$. If $c \neq 0$, then $c_{p}' = c \frac{\lambda_{+}^{p-k} - \lambda_{-}^{p-k}}{\lambda_{+} - \lambda_{-}}$, which grows exponentially fast with $p$ and thus $h_{0}' \not \in X_{c}^{k}$.

If $c_{k}' \neq 0$, by normalizing we can take $c_{k}' = 1$. Let $c_{k+1}' = c$. Then the coefficients $c_{p}'$ are given for $p \geq k$ by

\begin{align*}
 c_{p}' = \frac{\lambda_{-} - c}{\lambda_{-} - \lambda_{+}}\lambda_{+}^{p-k} + \frac{c - \lambda_{+}}{\lambda_{-} - \lambda_{+}}\lambda_{-}^{p-k}
\end{align*}

which again behave exponentially with $p$, implying that $h_{0}' \not \in X_{c}^{k}$, unless $c = \lambda_{-}$. However, in that particular case the coefficients $c_{p}'$ are, for all $p$, multiples of all the previous $c_{p}$ found before. This concludes the uniqueness.

\end{proof}

\begin{rem}
\label{remarkfilanocero}
 We remark that $c_{k-1} - 2z c_{k} + c_{k+1} \neq 0$, since

\begin{align*}
 c_{k-1} - 2z c_{k} + c_{k+1} & = 
\lambda_{+}^{p-1}(1+\lambda_{+}\lambda_{-} - 2z\lambda_{+}) - \lambda_{-}^{p-1}(1+\lambda_{-}^{2} - 2z\lambda_{-}) \\
& = \lambda_{+}^{p}(\lambda_{-} - \lambda_{+}) \neq 0.
\end{align*}

\end{rem}

\begin{lemma}
\label{lemaKn}
 Let $k>2$, and let $r_k$ and $K_{n}$ be defined as before. Then $K_{n}(r_k) \sim n$ as $n \to \infty$.
\end{lemma}

\begin{proof}
 The proof follows easily from the fact that the only dependence in $n$ is through the terms $2nr_k - (1+r_{k})^{2}\left(\frac{1-r_{k}}{1+r_{k}}\right)^{n}$ that appear in the bracket. The second term goes to zero since $\frac{1-r(k)}{1+r(k)} < 1$, hence the dominant one is $O(n)$.
\end{proof}

\begin{lemma}
Let $k>2$. Then $Y^{k-1}_{c}$/Range($DF(r(k),0)$) is one-dimensional.
\end{lemma}

\begin{proof}

Again, we will assume for simplicity that $k$ is even. The odd case is treated similarly. We claim that:

\begin{align*}
 \text{Range($DF(r(k),0)$)} = \left\{f(x) \in Y^{k-1}_{c} \left| f(x) = \sum_{m=1,m\neq k} c_{m}\sin(2mx)\right.\right\} \equiv Z
\end{align*}

If we manage to prove the claim, then it is evident that the codimension is 1.

\underline{Range($DF(r(k),0)$) $\subset Z$:}

Let $h \in X^{k}_{c}$, given by $h(x) = \sum_{m=1}^{\infty} a_{m}\cos(2mx)$. Then:

\begin{align*}
 DF(r(k),0)[h] = \sum_{m=1}^{\infty}K_{m}(r(k))\left(a_{m-1} - 2\left(\frac{1+r(k)}{1-r(k)}\right)a_{m} + a_{m+1}\right) \sin(2mx)
= \sum_{m=1}^{\infty}b_{m} \sin(2mx)
\end{align*}

Using the inequality between the arithmetic and the quadratic mean:

\begin{align*}
 |b_{m}|^{2} \leq 3K_{m}(r(k))^{2}\left(|a_{m-1}|^{2} + 4\left(\frac{1+r(k)}{1-r(k)}\right)^{2}|a_{m}|^{2} + |a_{m+1}|^{2}\right),
\end{align*}

which implies 

\begin{align*}
\|DF(r(k),0)[h]\|_{Y^{k-1}_{c}}^{2} & = \sum_{m=1}^{\infty}|b_{m}|^{2}(1+m)^{2k-2}(\cosh(cm)^{2} + \sinh(cm)^{2}) \\
& \leq C\sum_{m=1}^{\infty}|a_{m}|^{2}\left(K_{m-1}(r(k))^{2} + K_{m}(r(k))^{2} + K_{m+1}(r(k))^{2} \right)(1+m)^{2k-2}(\cosh(cm)^{2} + \sinh(cm)^{2}) \\
& \leq C\|h\|_{X^{k}_{c}}^{2} < \infty,
\end{align*}

where in the last inequality we have used the asymptotic growth of $K_{m}(r(k))$ given by Lemma \ref{lemaKn}. This shows that $DF(r(k),0)[h] \in Y^{k-1}_{c}$. The condition on the $k$-th coefficient follows from the fact that $K_{k}(r(k)) = 0$.

\underline{Range($DF(r(k),0)$) $\supset Z$:}

Let $y(x) = \sum_{m=1,m\neq k}c_{m}\sin(2mx)$. We want to find a preimage $h(x) = \sum_{m=1}a_{m}\cos(2mx) \in X^{k}_{c}$ such that $DF(r(k),0)[h] = y$. Let $\tilde{c}_{m} = \frac{c_{m}}{K_{m}(r(k))}$, with the convention that $\tilde{c}_{k} = 0$ and let $w_{m} = \frac{\lambda_{+}^{m} - \lambda_{-}^{m}}{\lambda_{+} - \lambda_{-}}$.

We define the $a_{m}$ in the following way:

\begin{align*}
 a_{m} = \left\{
\begin{array}{cc}
 \text{The solution to the system \eqref{sistemapreimagen}} & \text{ if } m < k \\
0 & \text{ if } m = k \\
 \displaystyle - \sum_{j=1}^{m-k-1}\tilde{c}_{j+k}\frac{w_{j}}{\lambda_{+}^{m-k}} - \sum_{j=m-k}^{\infty} \tilde{c}_{j+k}\frac{w_{m-k}}{\lambda_{+}^{j}}& \text{ if } m > k, \\
\end{array}
\right\}
\end{align*}

where the system \eqref{sistemapreimagen} is given by:

\begin{align}
\label{sistemapreimagen}
\left(
\begin{array}{cccccc}
-2z & 1 & 0 & 0 & 0 & 0 \\
1 & -2z & 1 & 0 & 0 & 0 \\
0 & 1 & -2z & 1 & 0 & 0  \\
0 & 0 & \ddots & \ddots & \ddots & 0 \\
0 & 0 & 0 & 1 & -2z & 1  \\
0 & 0 & 0 & 0 & 1 & -2z   \\
\end{array}
\right)
\left(
\begin{array}{c}
a_{1} \\
a_{2} \\
a_{3} \\
a_{4} \\
\vdots \\
a_{k-1}
\end{array}
\right)
= 
\left(
\begin{array}{c}
\tilde{c}_{1} \\
\tilde{c}_{2} \\
\tilde{c}_{3} \\
\tilde{c}_{4} \\
\vdots \\
\tilde{c}_{k-1}
\end{array}
\right)
\end{align}

The system \eqref{sistemapreimagen} has a (unique) solution since its eigenvalues are given by $-2z + \cos\left(\frac{h\pi}{k}\right), h = 1,\ldots, k-1$ (see \cite{Smith:numerical-solution-pde}) and none of them is equal to zero. Moreover, we have that for $m = k+1$:

\begin{align*}
a_{m-1} - (2z) a_{m} + a_{m+1} & = (-2z)a_{k+1} + a_{k+2} \\
& = (2z)\left(\sum_{j=1}^{\infty} \tilde{c}_{j+k} \frac{1}{\lambda_{+}^{j}}\right)  - \sum_{j=2}^{\infty} \tilde{c}_{j+k} \frac{2z}{\lambda_{+}^{j}} - \tilde{c}_{k+1} \frac{1}{\lambda_{+}^{2}} \\
& = \tilde{c}_{k+1}\left(\frac{2z}{\lambda_{+}} - \frac{1}{\lambda_{+}^{2}}\right) = \tilde{c}_{k+1}
\end{align*}

and for $m > k+1$:

\begin{align*}
a_{m-1} - (2z) a_{m} + a_{m+1} & = 
- \sum_{j=1}^{m-k-2}\tilde{c}_{j+k}\frac{w_{j}}{\lambda_{+}^{m-1-k}}(1-\frac{2z}{\lambda_{+}} + \frac{1}{\lambda_{+}^{2}})
- \sum_{j=m+1-k}^{\infty} \tilde{c}_{j+k}\frac{1}{\lambda_{+}^{j}}(w_{m+1-k}-(2z)w_{m+k}+w_{m-1-k}) \\
& - \tilde{c}_{m-1}\left(\frac{w_{m-1-k}}{\lambda_{+}^{m-1-k}} - 2z\frac{w_{m-1-k}}{\lambda_{+}^{m-k}} + \frac{w_{m-1-k}}{\lambda_{+}^{m+1-k}}\right) - \tilde{c}_{m}\left(\frac{w_{m-1-k}}{\lambda_{+}^{m-k}} - (2z)\frac{w_{m-k}}{\lambda_{+}^{m-k}} + \frac{w_{m-k}}{\lambda_{+}^{m+1-k}}\right) \\
& = \tilde{c}_{m}
\end{align*}


We are finally left to check that our solution belongs to the space $X^{k}_{c}$. To do so, it is enough to check the $l^{2}$ summability of the coefficients $a_{m+k}$ for $m > 0$. 

We have that:

\begin{align*}
|a_{m+k}| \lesssim \sum_{j=1}^{\infty}|\tilde{c}_{j+k}|\lambda_{-}^{|m-j|}
\lesssim \left(\sum_{j=1}^{\infty}|\tilde{c}_{j+k}|^{2}\lambda_{-}^{|m-j|}\right)^{1/2}\underbrace{\left(\sum_{j=1}^{\infty}\lambda_{-}^{|m-j|}\right)^{1/2}}_{\leq C},
\end{align*}

where we used the Cauchy-Schwarz inequality in the second inequality. Summing over $m$:

\begin{align*}
\sum_{m=1}^{\infty}|a_{m+k}|^{2} \lesssim \sum_{m=1}^{\infty} \left(\sum_{j=1}^{\infty}|\tilde{c}_{j+k}|^{2}\lambda_{-}^{|m-j|}\right)
= \sum_{j=1}^{\infty}|\tilde{c}_{j+k}|^{2} \sum_{m=1}^{\infty}\lambda_{-}^{|m-j|} \lesssim \sum_{j=1}^{\infty}|\tilde{c}_{j+k}|^{2} < \infty,
\end{align*}

proving thus the $l^{2}$-summability. This concludes the proof.

\end{proof}

\subsection{Step 5}

\begin{lemma}
Let $k > 2$. Then:

$\pa_{r} DF(r(k),0)(h_0) \not \in$ Range($DF(r(k),0)$), where Ker$(DF(r(k),0)) = \langle h_0 \rangle$.
\end{lemma}

\begin{proof}
To do so, we will show that the $k$-th sine component of $\pa_{r} DF(r(k),0)(h_0)$ is non-zero. Again, let us assume that $k$ is even. Let

\begin{align*}
 h_{0}(x) = \sum_{p=1}^{\infty}c_{p}\cos(2px),
\end{align*}

with the coefficients $c_{p}$ given in Lemma \ref{lemanucleo}. Then, the $k$-th sine component is given by

\begin{align*}
& \left. \pa_{r} \left(K_{k}(r)\left(c_{k-1} - 2\left(\frac{1+r}{1-r}\right) c_{k} + c_{k+1}\right)\right)\right|_{r=r(k)} \\
& = \underbrace{K_{k}'(r(k))}_{\neq 0}\underbrace{(c_{k-1} - 2\left(\frac{1+r(k)}{1-r(k)}\right) c_{k} + c_{k+1}))}_{\neq 0}
+ \underbrace{K_{k}(r(k))}_{=0}\left.\pa_{r}\left(- 2\left(\frac{1+r}{1-r}\right) c_{k}\right)\right|_{r=r(k)} \neq 0,
\end{align*}

where in the first inequality we used the fact that $r(k)$ is a simple root of $K_{k}(r)$ and in the second we used Remark \ref{remarkfilanocero}.

\end{proof}

\begin{rem}
Since the curvature of the patch can be bounded by $F(r) - C \varepsilon^{2}$, where $\varepsilon$ is the radius of the neighbourhood of the perturbations $R(x)$ and $F(r) > 0$ we can conclude that for a small enough subset of our solutions, the boundary of the patches is convex.
\end{rem}

\subsection{Step 6}

To show that $F$ maps $X^{k,odd}_{c}$ into $Y^{k-1,odd}_{c}$ (and $X^{k,even}_{c}$ into $Y^{k-1,even}_{c}$) respectively, it is enough to observe that if $R(x)$ is even and $R(x) = R(\pi-x)$ (resp. $R(x) = -R(\pi-x)$) then $F(r,R)(x)$ is odd and $F(r,R)(x) = -F(r,R)(\pi-x)$ (resp. $F(r,R)(x) = F(r,R)(\pi-x)$).

\end{proofthm}

\section{Existence for the perturbation from disks}

\label{sectiondiscos}

This section is devoted to show Theorem \ref{teoremaexistenciadiscos}.

\begin{proofthm}{teoremaexistenciadiscos}
The proof of this theorem follows the same steps that the proof of Theorem \ref{teoremaexistenciaelipses}. It will be divided into 6 steps. These steps correspond to check the hypotheses of the Crandall-Rabinowitz theorem \cite{Crandall-Rabinowitz:bifurcation-simple-eigenvalues} for $$F(\Omega,R)=\Omega R'-\sum_{i=1}^3F_i(R),$$
where, if $\al > 0$:

\begin{align*}
F_1(R)=&\frac{C(\alpha)}{R(x)}\int\frac{\sin(x-y)}{\left(\left(R(x)-R(y)\right)^2+4R(x)R(y)\sin^2\left(\frac{x-y}{2}\right)\right)^\frac{\alpha}{2}}
\left(R(x)R(y)+R'(x)R'(y)\right)dy,\\
F_2(R)=&C(\alpha)\int\frac{\cos(x-y)}{\left(\left(R(x)-R(y)\right)^2+4R(x)R(y)\sin^2\left(\frac{x-y}{2}\right)\right)^\frac{\alpha}{2}}
\left(R'(y)-R'(x)\right)dy,\\
F_3(R)=&C(\alpha)\frac{R'(x)}{ R(x) }\int\frac{\cos(x-y)}{\left(\left(R(x)-R(y)\right)^2+4R(x)R(y)\sin^2\left(\frac{x-y}{2}\right)\right)^\frac{\alpha}{2}}
\left(R(x)-R(y)\right)dy,
\end{align*}

and if $\al = 0$:

\begin{align*}
F_1(R)=&\frac{1}{4\pi R(x)}\int \sin(x-y)\log\left(\left(R(x)-R(y)\right)^2+4R(x)R(y)\sin^2\left(\frac{x-y}{2}\right)\right)
\left(R(x)R(y)+R'(x)R'(y)\right)dy,\\
F_2(R)=&\frac{1}{4\pi}\int \cos(x-y)\log\left(\left(R(x)-R(y)\right)^2+4R(x)R(y)\sin^2\left(\frac{x-y}{2}\right)\right)
\left(R'(y)-R'(x)\right)dy,\\
F_3(R)=&\frac{R'(x)}{ 4\pi R(x) }\int\cos(x-y)\log\left(\left(R(x)-R(y)\right)^2+4R(x)R(y)\sin^2\left(\frac{x-y}{2}\right)\right)
\left(R(x)-R(y)\right)dy,
\end{align*}

The hypotheses are the following:
\begin{enumerate}
\item The functional $F$ satisfies $$F(\Omega,R)\,:\, \R\times \{1+V^r\}\mapsto Y^{k}_{c},$$ where $V^r$ is the open neighborhood of 0
\begin{align*}
V^r & = 
\left\{
\begin{array}{cc}
\{ f\in X^{k+1}\,:\, ||f||_{X^{k+1}}<r\} & \text{ if } 0 < \al < 1 \\
\{ f\in X^{k+1+\log}\,:\, ||f||_{X^{k+1+\log}}<r\} & \text{ if } \al = 1 \\
\{ f\in X^{k+\al}\,:\, ||f||_{X^{k+\al}}<r\} & \text{ if } 1 < \al < 2 \\
\end{array}
\right.
\end{align*}
for  $0<r<1$ and $k\geq 3$.
\item $F(\Omega,1) = 0$ for every $\Omega$.
\item The partial derivatives $F_{\Omega}$, $F_{R}$ and $F_{R\Omega}$ exist and are continuous.
\item Ker($\mathcal{F}$) and $Y^{k}_{c}$/Range($\mathcal{F}$) are one-dimensional, where $\mathcal{F}$ is the linearized operator around the disk $R = 1$ at $\Omega = \Omega_m$.
\item $F_{\Omega R}(\Omega_m,1)(h_0) \not \in$ Range($\mathcal{F}$), where Ker$(\mathcal{F}) = \langle h_0 \rangle$.
\item Step 1 can be applied to the $m$-fold symmetric spaces.
\end{enumerate}

Since the majority of the proof was already done in \cite{Castro-Cordoba-GomezSerrano:existence-regularity-gsqg}, we will only outline this case. Steps 2, 4, 5 and 6 follow verbatim, and steps 1 and 3 can easily be adapted to the new spaces $X_{c}$. 

\end{proofthm}

\appendix

\section{Basic integrals}

\begin{lemma}
\begin{align*}
 \int_{0}^{2\pi}\frac{e^{iky}}{(1+r^2)+(r^2-1)\cos(y)}dy = \frac{\pi}{r}\left(\frac{1-r}{1+r}\right)^{|k|}
\end{align*}
\end{lemma}

\begin{proof}
We assume first $k \geq 0$. Then, performing the change of variables $e^{iy} = w$, we have that:

\begin{align*}
\int_{0}^{2\pi}\frac{e^{iky}}{(1+r^2)+(r^2-1)\cos(y)}dy = \frac{1}{r^2-1}\int_{0}^{2\pi} \frac{e^{iky}}{\cos(y)-\frac{1+r^2}{1-r^{2}}}dy = \frac{1}{r^2-1}\frac{2}{i}\int_{|w|=1} \frac{w^{k}}{w^{2}-2\frac{1+r^2}{1-r^{2}}w+1} dw.
\end{align*}

We look for the poles of the fraction (i.e. the zeros of the denominator). Solving the quadratic equation yields:

\begin{align*}
 w^{2} - 2\frac{1+r^2}{1-r^{2}}w + 1 = 0 \Rightarrow w = \left\{
\begin{array}{rl}
 w_{+} & = \frac{1+r^2}{1-r^{2}} + \sqrt{\left(\frac{1+r^2}{1-r^{2}}\right)^{2}-1} = \frac{1+r}{1-r} \\
 w_{-} & = \frac{1+r^2}{1-r^{2}} - \sqrt{\left(\frac{1+r^2}{1-r^{2}}\right)^{2}-1} = \frac{1-r}{1+r} \\
\end{array}
\right.
\end{align*}
%
%

Since only $w_{-}$ is in the interior of the region, we can apply the residue theorem to get

\begin{align*}
 \frac{1}{r^2-1}\frac{2}{i}\int_{|w|=1} \frac{w^{k}}{w^{2}-2aw+1} dw = \frac{1}{r^2-1} 2\pi i \frac{2}{i} \frac{w_{-}^{k}}{w_{-} - w_{+}}
= \frac{1}{r^2-1} 4\pi \left(\frac{1-r}{1+r}\right)^{k} \frac{r^2-1}{4r} = \frac{\pi}{r}\left(\frac{1-r}{1+r}\right)^{k}
\end{align*}

Assume now $k = -\tilde{k}, \tilde{k} \geq 0$. The integral becomes:

\begin{align*}
 \frac{1}{r^2-1}\frac{2}{i}\int_{|w|=1} \frac{1}{w^{\tilde{k}}}\frac{1}{w^{2}-2\frac{1+r^2}{1-r^{2}}w+1} dw
= \frac{1}{r^2-1}\frac{2}{i} \frac{1}{2\sqrt{\left(\frac{1+r^2}{1-r^{2}}\right)^{2}-1}}\int_{|w|=1} \frac{1}{w^{\tilde{k}}}\left(\frac{1}{w-w_{+}} - \frac{1}{w-w_{-}}\right) dw.
\end{align*}

There are now two poles of the function, one at $w = 0$ of multiplicity $\tilde{k}$, and a simple one at $w = w_{-}$. The contribution from $w_{-}$ is given by

\begin{align*}
 -\frac{1}{r^2-1}\frac{4\pi i}{2i} \frac{1}{\sqrt{\left(\frac{1+r^2}{1-r^{2}}\right)^{2}-1}} \frac{1}{w_{-}^{\tilde{k}}}.
\end{align*}

The contribution from 0 is

\begin{align*}
 \frac{1}{r^2-1}2\pi i\frac{2}{i}\frac{1}{2\sqrt{\left(\frac{1+r^2}{1-r^{2}}\right)^{2}-1}}\frac{1}{(k-1)!} \frac{d^{k-1}}{d^{k-1}w}\left.\left(\frac{1}{w-w_{+}} - \frac{1}{w-w_{-}}\right)\right|_{w=0}
= -\frac{1}{r^2-1}2\pi\frac{1}{\sqrt{\left(\frac{1+r^2}{1-r^{2}}\right)^{2}-1}}\left(\frac{1}{w_{+}^{\tilde{k}}} - \frac{1}{w_{-}^{\tilde{k}}}\right).
\end{align*}

Thus, the integral is equal to

\begin{align*}
 -\frac{1}{r^2-1}\frac{2\pi}{\sqrt{\left(\frac{1+r^2}{1-r^{2}}\right)^{2}-1}} \frac{1}{w_{+}^{\tilde{k}}} = \frac{\pi}{r}\left(\frac{1-r}{1+r}\right)^{\tilde{k}},
\end{align*}

as claimed.

\end{proof}

\begin{corollary}
\begin{align*}
 \int_{0}^{2\pi}\frac{\cos(ky)}{(1+r^2)+(r^2-1)\cos(y)}dy = \frac{\pi}{r}\left(\frac{1-r}{1+r}\right)^{|k|}, \quad 
 \int_{0}^{2\pi}\frac{\sin(ky)}{(1+r^2)+(r^2-1)\cos(y)}dy = 0
\end{align*}
\end{corollary}

\begin{lemma}
 Let $k \in \Z$. We have that, if $k \neq 0$:

\begin{align*}
 \int_{0}^{2\pi} e^{iky} \log\left(\left|\sin\left(\frac{y}{2}\right)\right|^{2}\right)dy = -2\pi\frac{1}{|k|}
\end{align*}

and if $k = 0$:

\begin{align*}
 \int_{0}^{2\pi} e^{iky} \log\left(\left|\sin\left(\frac{y}{2}\right)\right|^{2}\right)dy = -4\pi\log(2)
\end{align*}

\end{lemma}

\begin{proof}
 Assume $k \neq 0$. We have that

\begin{align*}
 \int_{0}^{2\pi} e^{iky} \log\left(\left|\sin\left(\frac{y}{2}\right)\right|^{2}\right)dy = -\frac{1}{i k} \int_{0}^{2\pi} \frac{e^{iky} - 1}{1-\cos(y)}\sin(y) dy
= -\frac{1}{k} \int_{0}^{2\pi} \frac{e^{iky} - 1}{e^{iy} - 1}(e^{iy} + 1) dy
\end{align*}

If $k > 0$:

\begin{align*}
 -\frac{1}{k} \int_{0}^{2\pi} \frac{e^{iky} - 1}{e^{iy} - 1}(e^{iy} + 1) dy = 
-\frac{1}{k} \int_{0}^{2\pi} \sum_{j=0}^{k-1} e^{ijy} (e^{iy} + 1) dy = -\frac{2\pi}{k}
\end{align*}

If $k < 0$:

\begin{align*}
 -\frac{1}{k} \int_{0}^{2\pi} \frac{e^{iky} - 1}{e^{iy} - 1}(e^{iy} + 1) dy = 
\frac{1}{k} \int_{0}^{2\pi} \frac{e^{iky} - 1}{e^{-iy} - 1}(e^{-iy} + 1) dy = 
\frac{1}{k} \int_{0}^{2\pi} \sum_{j=0}^{k-1} e^{-ijy} (e^{-iy} + 1) dy = \frac{2\pi}{k} = -\frac{2\pi}{|k|}
\end{align*}

The case $k = 0$ can be found in \cite[Formula 4.225.3]{Gradshteyn-Ryzhik:table-integrals}.

\end{proof}

\begin{lemma}
Let $k$ be an integer. We have that, if $k \neq 0$:

\begin{align*}
 \int_{0}^{2\pi} e^{iky} \log\left(\frac{1+r^2}{1-r^2}-\cos(y)\right)dy = -\frac{2\pi}{|k|}\left(\frac{1-r}{1+r}\right)^{|k|}
\end{align*}

and if $k = 0$:

\begin{align*}
 \int_{0}^{2\pi} e^{iky} \log\left(\frac{1+r^2}{1-r^2}-\cos(y)\right)dy = -2\pi \log\left(2\frac{1-r}{1+r}\right)
\end{align*}

\end{lemma}

\begin{proof}
Assume $k \neq 0$. If we integrate by parts, we obtain

\begin{align*}
 \int_{0}^{2\pi} e^{iky} \log\left(\frac{1+r^2}{1-r^2}-\cos(y)\right)dy & =
-\frac{1}{ik}\int_{0}^{2\pi} e^{iky} \frac{\sin(y)}{\frac{1+r^2}{1-r^2}-\cos(y)}dy \\
& =
\frac{1}{2k}\int_{0}^{2\pi} e^{i(k+1)y} \frac{1}{\frac{1+r^2}{1-r^2}-\cos(y)}dy -
\frac{1}{2k}\int_{0}^{2\pi} e^{i(k-1)y} \frac{1}{\frac{1+r^2}{1-r^2}-\cos(y)}dy \\
& = \frac{1-r^2}{2k}\left(\frac{\pi}{r}\left(\frac{1-r}{1+r}\right)^{|k+1|}-\frac{\pi}{r}\left(\frac{1-r}{1+r}\right)^{|k-1|}\right) \\
& = \frac{1-r^2}{2|k|}\frac{\pi}{r}\left(\frac{1-r}{1+r}\right)^{|k|}\left(-\frac{4r}{r^2-1}\right) = -\frac{2\pi}{|k|}\left(\frac{1-r}{1+r}\right)^{|k|}\\
\end{align*}

The case $k = 0$ is covered in \cite[Formula 4.224.9]{Gradshteyn-Ryzhik:table-integrals}.

\end{proof}

\section*{Acknowledgements}

We thank Peter Constantin for pointing out to us the question about the bifurcation from ellipses. AC, DC and JGS were partially supported by the grant MTM2014-59488-P (Spain) and ICMAT Severo Ochoa project SEV-2011-0087. AC was partially supported by the Ram\'on y Cajal program RyC-2013-14317 and ERC grant 307179-GFTIPFD. JGS was partially supported by an AMS-Simons Travel Grant.

 \bibliographystyle{abbrv}
 \bibliography{references}

%

\begin{tabular}{l}
\textbf{Angel Castro}   \\
{\small Departamento de Matem\'aticas}  \\
{\small Facultad de Ciencias}  \\
{\small Universidad Aut\'onoma de Madrid}  \\
{\small Campus Cantoblanco UAM, 28049 Madrid}  \\
 \\
\\
{\small Instituto de Ciencias Matem\'aticas, CSIC-UAM-UC3M-UCM}  \\
{\small C/ Nicol\'as Cabrera 13-15}  \\
{\small Campus Cantoblanco UAM, 28049 Madrid}  \\
{\small Email: angel\_castro@icmat.es}  \\
\\
  \textbf{Diego C\'ordoba} \\
  {\small Instituto de Ciencias Matem\'aticas} \\
 {\small Consejo Superior de Investigaciones Cient\'ificas} \\
 {\small C/ Nicolas Cabrera, 13-15, 28049 Madrid, Spain} \\
  {\small Email: dcg@icmat.es} \\
\\
 {\small Department of Mathematics} \\
 {\small Princeton University} \\
 {\small 804 Fine Hall, Washington Rd,} \\
  {\small Princeton, NJ 08544, USA} \\
 {\small Email: dcg@math.princeton.edu} \\
\\
\textbf{Javier G\'omez-Serrano} \\
{\small Department of Mathematics} \\
{\small Princeton University}\\
{\small 610 Fine Hall, Washington Rd,}\\
{\small Princeton, NJ 08544, USA}\\
 {\small Email: jg27@math.princeton.edu}
  \\

\end{tabular}

\end{document}